\documentclass[11pt]{amsart}
\usepackage[utf8x]{inputenc} 
\usepackage[a4paper]{geometry} 
\usepackage{enumerate}  
\usepackage{amsmath} 
\usepackage{amssymb}
\usepackage{color}
\usepackage{multirow}
\usepackage{hyperref}  
\usepackage{verbatim}
\usepackage{tikz}
\usepackage{tkz-euclide} 
\usepackage[normalem]{ulem}
\usepackage{comment}
                 \hypersetup{ pdfborder={0 0 0}, 
                              colorlinks=true, 
                              linktoc=page,
                              pdfauthor={Michael Hoff, Isabel Stenger and José Ignacio Ya\~n ez}, 
                              pdftitle={}} 
\definecolor{Gray}{gray}{0.9}                            
\usepackage[color, arrow,curve,matrix,tips,frame]{xy}
\usepackage{graphicx}   
\usepackage[percent]{overpic}

\theoremstyle{plain} 
\newtheorem{proposition}{Proposition}[section] 
\newtheorem{theorem}[proposition]{Theorem} 
\newtheorem{conjecture}[proposition]{Conjecture} 
\newtheorem{lemma}[proposition]{Lemma}

\theoremstyle{definition} 
\newtheorem{definition}[proposition]{Definition} 
\newtheorem{notation}[proposition]{Notation}
\newtheorem{construction}[proposition]{Construction}
\newtheorem{question}[proposition]{Question}
  
\newtheorem{example}[proposition]{Example} 

\newtheorem{remark}[proposition]{Remark}

\newcommand{\ZZ}{{\mathbb{Z}}} 
\newcommand{\RR}{{\mathbb{R}}}    
\newcommand{\CC}{{\mathbb{C}}}      
\newcommand{\PP}{{\mathbb{P}}}

\newcommand{\Per}{\mathrm{Per}}
\newcommand{\ww}{\mathbf{w}}
\newcommand{\MovE}{\overline{\operatorname{Mov}}^e}
\newcommand{\NefE}{\operatorname{Nef}^e}

\providecommand{\Pic}{\mathop{\rm Pic}}

\DeclareMathOperator{\Amp}{Amp}
\DeclareMathOperator{\Nef}{Nef}
\DeclareMathOperator{\Eff}{Eff}
\DeclareMathOperator{\Bir}{Bir}

\DeclareMathOperator{\Aut}{Aut}
\DeclareMathOperator{\Int}{int}
\DeclareMathOperator{\id}{id}

\DeclareMathOperator{\MovC}{\overline{Mov}}
\DeclareMathOperator{\EffC}{\overline{Eff}}
\DeclareMathOperator{\Mov}{Mov}

\DeclareMathOperator{\PGL}{PGL}

\numberwithin{equation}{section}

\makeatletter
\newcommand{\bigsymbol}[1]{%
  \DOTSB
  \mathop{
    \mathchoice{\big@symbol\displaystyle\Large{#1}}
               {\big@symbol\textstyle\large{#1}}
               {\big@symbol\scriptstyle\normalsize{#1}}
               {\big@symbol\scriptscriptstyle\small{#1}}%
    }\slimits@
}

\newcommand{\big@symbol}[3]{%
  \vcenter{%
    \sbox\z@{$#1\sum$}%
    \dimen@=0.875\dimexpr\ht\z@+\dp\z@\relax
    #2
    \hbox{\resizebox{!}{\dimen@}{$\m@th#3$}}%
  }%
  \vphantom{\sum}%
}
\makeatother

\newcommand{\bigast}{\bigsymbol{*}}

\title[Movable cones of complete intersections]{Movable cones of complete intersections of multidegree one on products of projective spaces}

\author[M. Hoff]{Michael Hoff} 
\author[I. Stenger]{Isabel Stenger} 
\author[J.I.  Y\'a\~nez]{José Ignacio Y\'a\~nez}
\email{\href{mailto:hahn@math.uni-sb.de}{hahn@math.uni-sb.de}}
\address{Universit\"at des Saarlandes, Campus E2 4, D-66123 Saarbr\"ucken, Germany} 
\email{\href{mailto:stenger@math.uni-sb.de}{stenger@math.uni-sb.de}} 
\address{Department of Mathematics, University of Utah, 155 South 1400 East, Salt Lake City, UT}
\email{\href{mailto:yanez@math.uni.edu}{yanez@math.utah.edu}}

\date{\today} 

\begin{document}

\begin{abstract} We study Calabi-Yau manifolds which are complete intersections of hypersurfaces of multidegree $1$ in an $m$-fold product of $n$-dimensional projective spaces. 
Using the theory of Coxeter groups, we show that the birational automorphism group of such a Calabi-Yau manifold $X$ is infinite and a free product of copies of $\ZZ$ . Moreover, we give an explicit description of the boundary of the movable cone $\MovC(X)$. 
In the end, we consider examples for the general and non-general case and picture the movable cone and the fundamental domain for the action of $\Bir(X)$.
\end{abstract}

\maketitle

\section{Introduction}

The Kawamata--Morrison cone conjecture predicts that the movable cone $\MovC(X)$ of a Calabi-Yau manifold $X$ is rational polyhedral up to the action of the  birational automorphism group $\Bir(X)$. 
Even though it has been proved in several instances (see \cite{Kaw97}, \cite{Tot10}, \cite{Pen12}, \cite{LP13}, \cite{Og14}, \cite{CO15}, \cite{Y21}, \cite{Wang21}, \cite{HT18}, \cite{LW21}, \cite{ILW22} and references therein), the conjecture remains widely open. We refer to \cite{LOP20} for a good survey on this problem.

The main focus of this article are the following varieties: let $\PP := \PP^n \times \cdots \times \PP^n$ be an $m$-fold product of projective spaces, with $n \geq 1$ and $nm - (n+1) \geq 3$. Consider $X$ to be a general complete intersection subvariety given by the intersection of $n+1$ ample divisors of multidegree $(1,\ldots ,1)$. By generality and adjunction we have that $X$ is a smooth Calabi-Yau variety. Moreover, $X$ has only finitely many minimal models $X=X_0,X_1\ldots, X_m$ up to isomorphism and satisfies the Kawamata-Morrison cone conjecture for the movable cone (see \cite{Wang21}). However, the structures of the movable cone $\MovC(X)$ and the birational automorphism group $\Bir(X)$ are not known.

In this paper,  we show that in contrast to the complete intersection Calabi-Yau manifolds considered in \cite{CO15} and \cite{Y21}, the birational automorphism group is \textbf{not} a Coxeter group in this case and give a complete description of $\Bir(X)$.

\begin{theorem}
Let $\PP =\PP^n \times \cdots \times \PP^n$ be an $m$-fold product and consider $X$ a general complete intersection of $n+1$ hypersurfaces of multidegree $(1,\ldots, 1)$ in $\PP$. 
\begin{itemize}

\item[(1)] The automorphism group $\Aut(X)$ is finite and acts trivially on $N^1(X)_\RR$. If $n \geq 3$ and $m\geq 3$, then $\Aut(X)$ is trivial  for $X$ very general.

\item[(2)] If $n \geq 2$, the birational automorphism group $\Bir(X)$ is isomorphic to the group $\underbrace{\ \ZZ \ast \cdots \ast \ZZ\ }_{{m\choose 2} \text{ terms}}$ up to the finite kernel $\Aut(X)$. 
\end{itemize}
\label{theorem_mainbir}
\end{theorem}

Recall that Wang showed that a general $X$ has exactly $m+1$ minimal models $X=X_0,X_1, \ldots, X_m$ and that the cone $\Nef^e(X_i)$ for $i=0,\ldots,m$  is a rational polyhedral cone in $N^1(X_i)_\RR \cong \RR^m$ spanned by $m$ extremal rays (see \cite{Wang21}). 
Using results from \cite{Loo14}, Wang concluded that $X$ satisfies the movable cone conjecture. 
We refine this statement by providing an explicit  description of a fundamental domain $\Pi$ for the action of $\Bir(X)$ on the effective movable cone $\MovE(X)$.
\begin{theorem}\label{thm_mainpi}
Let $\PP =\PP^n \times \cdots \times \PP^n$ be an $m$-fold product, and let $X$ be a general complete intersection of $n+1$ hypersurfaces of multidegree $(1,\ldots, 1)$ in $\PP$ with minimal models $X_0 = X, X_1,\ldots,X_m$. Then there exist birational maps $\varphi_{0,i}\colon X \dashrightarrow X_i$ for $i=1,\ldots,m$ such that 
\[
\Pi = \Nef^e(X_0) \cup \bigcup_{i=1,\ldots,m} \varphi_{0,i}^*\Nef^e(X_i) \subset \MovE(X)
\]
is a fundamental domain for the action of $\Bir(X)$ on $\MovE(X)$. 
\end{theorem}
Consequently, the fundamental domain $\Pi$ is the union of $m+1$ polyhedral cones, where each pullback of $\Nef^e(X_i)$ for $i=1,\ldots,m$ is glued to a codimension one face of $\Nef^e(X_0)$. 

 For our last main theorem we recall that a Coxeter group $W$ is, roughly speaking, a group generated by a finite number of reflections, and it can be represented as matrices acting on a real dimensional vector space whose dimension is equal to the number of generators. We can associate a rational polyhedral cone $D$ to the group $W$, the $W$-orbit of $D$ is a convex cone, called the Tits cone, and the cone $D$ is a fundamental domain for the action of $W$ on the Tits cone. In \cite{CO15} and \cite{Y21} they identify $\Bir(X)$ with a Coxeter group $W$, such that under this connection the closure of the Tits cone of $W$ corresponds to $\MovC(X)$ and the fundamental domain $D$ corresponds to the nef cone $\Nef(X)$.
As $\Bir(X)$ is not a Coxeter group in the case of Theorem \ref{theorem_mainbir}, this argument cannot be used directly. Nevertheless, we find that its elements can be related to a Coxeter system $(W,S)$. More precisely, we show that, up to a permutation, any birational automorphism of $X$ can be represented as a product of matrices $t_1,\ldots,t_m$ associated to $(W,S)$ (see Subsection \ref{subsection:CG} for the definitions). 
Using this together with the description of the fundamental domain $\Pi$ from Theorem \ref{thm_mainpi}, we explicitly describe  the boundary of  $\MovC(X)$ following  \cite{Y21}.

\begin{theorem} Let $\PP$ and $X$ be as above.  For $i =1,\ldots,m$, let $H_i$ be the pullback of a hyperplane class of the $i$-th factor $\PP^n$ to $X \subset \PP^n \times \cdots \times \PP^n$. Then the boundary of $\MovC(X)$ is the closure of the union of the $W$-orbits of the cones  $\{ a_\lambda v_\lambda + \sum_{k\neq i,j} a_k H_k \mid a_\lambda \geq 0, a_k\geq 0\}$, where $v_\lambda$ is:
	\begin{itemize}
		\item If $n \geq 3$, an eigenvector associated to the unique eigenvalue $\lambda >1$ of $t_it_j$; or
		\item If $n = 2$, $v_\lambda = 0$.
	\end{itemize}
\label{theorem_mainboundary}
\end{theorem}

Another emphasis of this paper is to consider explicit examples. Whereas the movable cone is well described for a Calabi-Yau manifold $X$ with infinite $\Bir(X)$ and Picard number $\rho(X) =2$ (see \cite{LP13}), the structure of the cone $\MovC(X)$ is poorly understood in higher Picard numbers. A class of concrete examples may help to get a better understanding of the general situation and can be used for determining divisors on the boundary of $\MovC(X)$ or $\EffC(X)$ where the numerical dimension behaves interestingly.  

In Section \ref{section_ex} we consider examples for different $n$ and $m$. On the one hand we provide  examples for the general case 
where we picture the movable cone together with its chamber structure by the action of $\Bir(X)$; and by using the description of the boundary from Theorem \ref{theorem_mainboundary} (see Examples \ref{examp_3foldn2m3} and \ref{examp_3foldn3m3} in Subsection \ref{subsection_general}). On the other hand, in Subsection \ref{subsection_nongeneral} we give examples where the automorphism group of $X$ is non-trivial which are not considered in \cite{Wang21}.  In particular, we notice that if the automorphism group does not act trivially on $N^1(X)_\RR$, then the structure of the movable cone is significantly different to the general case. This is shown in Example \ref{example_nontrivialaut}  where we verify that the movable cone conjecture is satisfied in this case and provide a conjectural description of the pseudoeffective cone.

Having an explicit description of both the group of birational automorphisms of $X$ and of the movable cone $\MovC(X)$ has proven to be useful to compute the numerical dimension of divisors on the boundary of $\MovC(X)$ (see for example \cite{L21}, \cite{HS22}, \cite{Y21}). One important detail of these computations is that $\MovC(X) = \EffC(X)$, so the numerical dimension of a divisor on the boundary can be computed by pulling it back to a nef divisor via a birational automorphism of $X$. In the case of Example \ref{example_nontrivialaut}, we have that the pseudoeffective cone is not equal to the movable cone, so this exact method cannot be used. In Proposition \ref{prop_numDim} we compute the numerical dimension of most of the expected boundary of $\EffC(X)$. See also Question \ref{question_numDim} and the discussion preceding it.

\subsection*{Acknowledgments} 
We would like to thank Vladimir Lazi\'c and Christopher Hacon for useful discussions. 
The first two authors were supported by the Deutsche Forschungsgemeinschaft (DFG, German Research Foundation) - Project-ID 286237555 - TRR 195. The third author was partially supported by NSF research grants no:  DMS-1952522, DMS-1801851 and by a grant from the Simons Foundation; Award Number: 256202. 

\section{Preliminaries}	

\subsection{Coxeter groups} \label{subsection:CG}

For a complete introduction to the theory of Coxeter groups we refer to \cite{Humphreys90}. Let $W$ be a finitely generated group and let $S = \{s_i\}_{i=1}^n$ be a finite set of generators. We say that the pair $(W,S)$ is a \emph{Coxeter system} if there are integers $m_{ij} \in \ZZ_{\geq 0} \cup \{\infty\}$ such that \begin{itemize}
\item $W = \langle s_i \mid (s_is_j)^{m_{ij}} = 1\rangle$,
\item $m_{ii} = 1$ for all $i$,
\item $m_{ij} = m_{ji} \geq 2$ or $= \infty$.
\end{itemize}

A group is called a \emph{Coxeter group} if there is a finite subset $S\subset W$ such that $(W,S)$ is a Coxeter system.

Notice that if $m_{ij} = \infty$ for all distinct $i$ and $j$, then \[
W \cong \underbrace{\ZZ/2\ZZ \ast \cdots \ast \ZZ/2\ZZ}_{n \text{ times}}
 \] and is sometimes referred as the \emph{universal Coxeter group of rank $n$.}

To a given Coxeter system $(W,S)$ we can associate a real vector space $V$ of dimension $n = |S|$ with basis $\Delta = \{\alpha_i\}_{i=1}^n$, and a bilinear form $\mathcal{Q}$. To define the bilinear form $\mathcal{Q}$, for each pair $(i,j)$ such that $m_{ij} = \infty$ we choose real numbers $c_{ij}$, with the condition that $c_{ij} = c_{ji} \geq 1$. Then the bilinear form is defined as
\[
\mathcal{Q}(\alpha_i,\alpha_j) = 
\begin{cases}
- \cos \left( \frac{\pi}{m_{ij}} \right) & \text{if } m_{ij} < \infty\\
-c_{ij} & \text{if } m_{ij} = \infty,
\end{cases}
\]
and we can encode the choice of numbers $c_{ij}$ in a matrix $Q$ with entries $Q_{ij} = \mathcal{Q}(\alpha_i,\alpha_j)$. Notice that by definition $m_{ii} = 1$, so $\mathcal{Q}(\alpha_i,\alpha_i) = -\cos \pi = 1$.

For each generator $s_i$ of the Coxeter system $(W,S)$ we can assign to it an element $\tau_i \in GL(V)$ such that 
\[
\tau_i(w) = w - 2\mathcal{Q}(w,\alpha_i) \alpha_i.
\]
for $w \in V$. 
\begin{proposition}
The homomorphism $\rho \colon W \to GL(V)$ defined by $s_i \mapsto \tau_i$ is a faithful geometric representation of the Coxeter group $W$.
\label{prop_faithful}
\end{proposition}

\begin{proof}
For the case when we choose $c_{ij} = 1$ for $m_{ij} = \infty$, see \cite[Section 5.4 Corollary]{Humphreys90}, and for the general case see \cite[Theorem 2(6)]{Vinberg71}.
\end{proof}

Given that the representation does depend on the choice of $Q$, we keep track of this data by denoting our Coxeter system $(W,S)_Q$.

In this article we are more interested of the action of $W$ on the dual space $V^*$. Define in $V^*$ the convex cone $D \subseteq V^*$ as the intersection of the half-spaces \[ D := \bigcap_{i=1}^n \{ f \in V^* \mid f(\alpha_i) \geq 0 \}. \]

Denote the dual representation of $W$ on $V^*$ as $\rho^*$. Then we definte the \emph{Tits cone $T$} as the $W$-orbit of $D$, i.e,
\[
T = \bigcup_{w \in W} \rho^*(w)(D).
\] 

\begin{proposition}[{\cite[Theorem 2 and Proposition 8]{Vinberg71}}]
The Tits cone $T$ is a convex cone, and $D$ is a fundamental domain for the action of $W$.
\end{proposition}

To connect the action of $W$ on $V$ and on $V^*$ we can identify $V^*$ with $V$ as follows. Let $\{\beta_i\}$ be the vectors on $V$ such that $\mathcal{Q}(\alpha_i,\beta_j) = \delta_{ij}$, where $\delta_{ij}$ is the Kronecker Delta. Under this identification, the fundamental domain $D$ is the cone 
\[ 
D = \left\{\sum_{i=1}^n a_i\beta_i \mid a_i \geq 0 \text{ for all } 1 \geq i \geq n \right\},
\]
and the matrices associated to the generators of the dual action of $W$, $\rho^*(s_i)$, with respect to the basis $\{\beta_i\}$ correspond to the transpose of the matrices associated to $\tau_i$  with respect to the basis $\{\alpha_i\}$. We call $t_i$ the matrix representing $\rho^*(s_i)$ with respect to the basis $\{\beta_i\}$. 

\begin{example}
Let $(W,\{s_1,s_2,s_3\})$ be a Coxeter system with $m_{ij} = \infty$ when $i \neq j$. Then $W \cong \ZZ/2\ZZ \ast \ZZ/2\ZZ \ast\ZZ/2\ZZ$. Choose the bilinear form $\mathcal{Q}$ with associated matrix
\[
Q = \begin{pmatrix}
1 & -2 & -2\\
-2 & 1 & -2\\
-2 & -2 & 1
\end{pmatrix}.
\]
Then with respect to the basis $\{\alpha_1,\alpha_2,\alpha_3\}$ the matrices associated to $\tau_1$, $\tau_2$ and $\tau_3$ are
\[
\begin{pmatrix}
-1 & 4 & 4 \\
0 & 1 & 0 \\
0 & 0 & 1
\end{pmatrix} ,
\begin{pmatrix}
1 & 0 & 0 \\
4 & -1 & 4 \\
0 & 0 & 1
\end{pmatrix} ,
\begin{pmatrix}
1 & 0 & 0 \\
0 & 1 & 0 \\
4 & 4 & -1
\end{pmatrix}
\]
respectively. Then the matrices $t_1$, $t_2$ and $t_3$ are 
\[
\begin{pmatrix}
-1 & 0 & 0 \\
4 & 1 & 0 \\
4 & 0 & 1
\end{pmatrix} , 
\begin{pmatrix}
1 & 4 & 0 \\
0 & -1 & 0 \\
0 & 4 & 1
\end{pmatrix} ,
\begin{pmatrix}
1 & 0 & 4 \\
0 & 1 & 4 \\
0 & 0 & -1
\end{pmatrix}.
\]
\end{example}

\begin{remark}\label{rem_AffSpace}
When studying the Tits cone $T$ it is convenient to work on $\PP V$. Choose $V_1$ an affine hyperplane such that each ray $\RR_{>0} \alpha_i$, for $1 \leq i \leq n$, intersects at one point and denote it $\widehat{\alpha_i}$. Let $\varphi$ be the linear form such that $V_1$ is $\{v \in V \mid\varphi(v) = 1 \}$. Define \[ \widehat{v} = \frac{v}{\varphi(v)},\ \text{for }v\in V\setminus \{\varphi(v)=0\}. 
\] These points, along with the hyperplane at infinity, form $\PP V$, and we will refer to points in $\PP V$ as \emph{directions} of $V$.
\end{remark}

To study the boundary of the closure of the Tits cone, we need some additional properties of $W$. We say that the Coxeter system $(W,S)_Q$ is \emph{Lorentzian} if the signature of the matrix $Q$ is $(n-1,1)$. Notice that this property depends on the choice of matrix $Q$ and it is not intrinsic to the Coxeter system $(W,S)$.

The main consequence of this property is that given an infinite reduced word $\ww = s_{k_1}s_{k_2}\cdots$ in $W$, any injective sequence $\{w_i\cdot \widehat{x}\}$, where $w_i = s_{k_1}s_{k_2}\cdots s_{k_i}$, converges to the same direction, regardless of $\widehat{x} \in \PP V$. Here, by \emph{injective sequence} we mean that $w_i \cdot \widehat{x} \neq w_j \cdot \widehat{x}$ as elements in $\PP V$, when $i \neq j$.

\begin{theorem}
Let $\ww$ be an infinite reduced word in a Lorentzian Coxeter system $(W,S)_Q$. Then every injective convergent sequence $\{w_i\cdot \hat{x}\}$ converges to the same unique direction $\widehat{\gamma}(\ww)$.
\label{thmInfiniteWord}
\end{theorem}

\begin{proof}
This result follow from \cite[Theorem 2.5, Corollary 2.6, Theorem 2.8 and Corollary 2.9]{ChenLabbe17}.
\end{proof}

\begin{remark}\label{rem_quadric}
It is interesting to notice that by \cite[Theorem 2.5]{ChenLabbe17}, after the identification of the dual space $V^*$ with $V$, all the accumulation points of $\MovC(X)$ lie on the quadratic cone defined by the quadric associated to the inverse of the symmetric matrix $Q$ of the Coxeter system $(W,S)_Q$. 
\end{remark}

\subsection{Calabi-Yau manifolds and main construction} \label{subsection:Construction}

We recall some definitions and results which will be used in this paper.
\begin{definition}
A \emph{Calabi-Yau} manifold $X$ of dimension $n$ is a projective manifold with trivial canonical class $K_X$  and $h^1(X,{\mathcal O}_X) =0$.

\end{definition}

Let $X$ be a projective manifold, and let $N^1(X)$ be the N\'eron–Severi group of $X$. We denote the rank of $N^1(X)$ by $\rho(X)$.  Inside $N^1(X)_\RR = N^1(X) \otimes \RR$ we have
\begin{itemize}
	\item the nef cone $\Nef(X)$, and its interior the ample cone $\Amp(X)$; 
		\item the effective nef cone $\NefE(X) = \Nef(X) \cap \Eff(X)$;
		\item the closure of the cone generated by movable  divisors $\MovC(X)$ and its interior $\Mov(X)$;
	\item the effective cone $\Eff(X)$ with closure $\EffC(X)$, the pseudoeffective cone; and
	\item the effective movable cone $\MovE(X) = \MovC(X)  \cap \Eff(X)$.
\end{itemize}
Furthermore, we denote by $\Aut(X)$ the automorphism group and by $\Bir(X)$ the birational automorphism group. 

The Kawamata--Morrison cone conjecture which connects the convex geometry of the nef cone $\Nef(X)$ and the movable cone $\MovC(X)$ of a Calabi-Yau manifold $X$ with the action of $\Aut(X)$ and  $\Bir(X)$, respectively, can be formulated as follows:
\begin{conjecture}[\cite{Mor93},\cite{Kaw97}] Let $X$ be a Calabi-Yau manifold. 
	\begin{itemize}
		\item[(1)] 	There exists a rational polyhedral cone $\Pi$ which is a fundamental domain for the action of $\Aut(X)$ on $\Nef(X) \cap \Eff(X)$, in the sense that
		\[
		\Nef(X) \cap \Eff(X) = \bigcup_{g\in \Aut(X)} g^* \Pi
		\] 
		with $\Int \Pi  \cap \Int g^* \Pi = \emptyset$ unless $g^* = \id$. 
		\item[(2)] There exists a rational polyhedral cone $\Pi'$ which is a fundamental domain for the action of $\Bir(X)$ on $\MovC(X) \cap \Eff(X)$.
	\end{itemize}
\end{conjecture}

Next we explain Wang's construction of Calabi-Yau manifolds as general complete intersections in a product of projective spaces. For more details and  proofs we refer to \cite{Wang21}.
\begin{notation}
We write $\PP = \PP^n_0 \times \PP^n_1 \times \ldots \times \PP^n_m$ for the $(m+1)$-fold product of $n$-dimensional projective spaces and assume that $nm - (n+1) \geq 3$. We denote by $\widehat{\PP}_i$ the $m$-fold product of $\PP^n$s, where the $i$-th factor is missing, and by $\widehat{\PP}_{i,j}$ the product where the $i$-th and the $j$-th factors are missing. If we denote by $H_i$  the pull-back of a hyperplane class in $\PP^n_i$, then $\Pic(\PP)$ is generated by $H_0,\ldots,H_m$ and 
		$$\Eff(\PP) = \Nef(\PP) = \Mov(\PP) = \textup{Cone}(H_0,\ldots,H_m)$$ where the last object is the convex cone in $N^1(\PP)_{\RR} \cong \RR^{m+1}$ generated by 
		$H_0,\ldots, H_m$. 
	\label{notation}
\end{notation}

We consider a general complete intersection $X$ given by $n+1$ hypersurfaces $D_j$ of multidegree $(1,\ldots,1) $ in $\widehat{\PP}_0$.  Then, by generality and adjunction, $X$ is a smooth Calabi-Yau variety of dimension $mn - n - 1.$
By the constraints on $n$ and $m$ we have that $\dim X \geq 3$. 
\begin{remark}
As the general complete intersection of $n$ hypersurfaces of multidegree $(1,\ldots,1)$ in $\widehat{\PP}_0$ is a smooth Fano manifold, by intersecting $n+1$ of such hypersurfaces and obtaining $X$, we know that 
\[ \Nef(X) \cong \Nef(\widehat{\PP}_0)\]
 and that the automorphism group $\Aut(X)$ is finite by \cite[Theorem 3.1]{CO15}.
\end{remark}
\begin{construction}[{\cite[Construction 7.2]{Wang21}}]\label{cons_Xi}

Let $\textbf{x}^i = [x_0^i,\ldots,x_n^i]$ be the homogeneous coordinates of $\PP^n_i$ for $i = 0,\ldots,m$. We denote by 
$$ I   = \{(r_1,\ldots,r_m) \mid 0 \leq r_k \leq n \textup{ for each }k\}.$$
We can write the $n+1$ hypersurfaces $D_j$ defining $X$ as
\[
\sum_{r=(r_1,\ldots,r_m) \in I} a_j^{r} x^1_{r_1} x^2_{r_2}\cdots x^m_{r_m}.
\]	
Moreover, setting 
$$ b^{r_0r_1\ldots r_m} := a_{r_0}^{r}$$ 
for each $r=(r_1,\ldots,r_m) \in I$ and $0 \leq r_0 \leq n$ we can define a $(1,\ldots,1)$-form 
\begin{equation}\label{eq_defF}
F = \sum_{0\leq r_k \leq n, \ 0 \leq k \leq m} b^{r_0r_1\ldots r_m}
	x^0_{r_0}x^1_{r_1} x^2_{r_2}\cdots x^m_{r_m}
\end{equation}
in $\PP = \PP^n_0 \times \PP^n_1 \times \ldots \times \PP^n_m$.
\end{construction}
Now, differentiating $F$ with respect to the $n+1$ variables of a $\PP^n_i$ gives $n+1$ hypersurfaces of multidegree $(1,\ldots,1)$ in $\widehat{\PP}_i$ which define a complete intersection variety $X_i$ with $X_0 = X$. Moreover, if we choose $X$ generically, the $X_i$ are also smooth and define Calabi-Yau varieties by \cite[Section 7]{Wang21} with 
$$\Nef(X_i) \cong \Nef(\widehat{\PP}_i) \cong \Nef(\widehat{\PP}_0) \cong \Nef(X_0). $$

For $i \neq j$, we denote the pullback of the hyperplane class of $\PP^n_j$ to $X_i$ by $H^i_j$. Then $\Nef(X_i)$ is the convex cone generated by the $m$ elements in 
\begin{equation}\label{eq_basisBi}
B_i = \{ H^i_0,\ldots,H^i_{i-1},H^i_{i+1}, \ldots, H^i_{m}\}
\end{equation}
which  form a basis of 
$N^1(X_i)_{\RR}  \cong \RR^m$. 

Two different $X_i$ and $X_j$ can be projected to the space $\widehat{\PP}_{i,j}$ and have the same image there which we denote by $W_{i,j}$. Indeed let $A_{i,j}$ be the $(n+1)\times (n+1)$-matrix obtained from $F$  by first differentiating $F$ with respect to the $n+1$ variables of $\PP^n_i$ and afterwards differentiating these $n+1$ forms with respect to the variables of  $\PP^n_j$.  Then $A_{i,j} = A_{j,i}^t$ and the image of $X_i$ (resp. $X_j$) to $\widehat{\PP}_{i,j}$ is the determinantal hypersurface
$$ W_{i,j} = (\det(A_{i,j}) = 0 ) \subset \widehat{\PP}_{i,j}$$ 
of multidegree $(n+1,\ldots,n+1)$ and dimension 
$$ (m-1)n - 1 = \dim X_i.$$
In the following, we denote the projection from $X_i$ to $W_{i,j} = W_{j,i}$ by $\pi^i_{i,j}$.  The projection $\pi^i_{i,j}$  has connected fibers and is an isomorphism outside the subset of $W_{i,j}$ consisting of points where the rank of $A_{i,j}$ is $\leq n-1$.
Thus, for a general choice of $X$, the maps $\pi^i_{i,j}$ are birational. 
Furthermore, the two projections $\pi^i_{i,j}$ and $\pi^j_{i,j}$ are simultaneously divisorial or small by \cite[Lemma 7.4]{Wang21}.

Finally, for each pair of $i\neq j$ there is an induced birational map $\varphi_{i,j}$:
\[
\begin{xy}
\xymatrix{
	X_i  \ar@{-->}[rr]^{\varphi_{i,j}} \ar[dr]_{\pi^i_{i,j}}& & X_j \ar[dl]^{\pi^j_{i,j}}
	\\
	& W_{i,j} &	
}
\end{xy}
\]
Moreover, if $F$ is chosen generically, then $\varphi_{i,j}$ is not an isomorphism over $W_{i,j}$ (see \cite[Assumption 7.5]{Wang21}) and the projections $\pi^i_{i,j}$ and $\pi^j_{i,j}$ are both small. In this case, $\varphi_{i,j}$ is the flop of $\pi^i_{i,j}$  (see \cite[Lemma 7.6]{Wang21}). The main theorem of \cite{Wang21} concerning this class of Calabi-Yau manifolds is the following theorem. 

\begin{theorem}[{\cite[Theorem 1.3]{Wang21}}]
	Let $X\subset \PP^n_1 \times \ldots \times \PP^n_m$ be a general complete intersection of $n+1$ hypersurfaces of multidegree $(1,\ldots, 1)$, and let $X_i \subset \widehat{\PP}_i$ for $i=1,\ldots,m$ as defined above. Then $X$ has only finitely many minimal models $X_0=X, X_1,\ldots,X_m$ up to isomorphism. Moreover, there exists a rational polyhedral cone $\Pi$  which is a fundamental domain for the action of $\Bir(X)$ on $\MovE(X)$.
	\label{theorem_wang}
\end{theorem}


\section{Computing Bir(X)} \label{section_computing}

Throughout this section we will assume that $\PP = \PP^n_0 \times \cdots \times \PP^n_m$ with $nm - (n+1) \geq 3$ and that $F$ is a general form of multidegree $(1,\ldots ,1)$ in $\PP$ as in Notation \ref{notation}. For $i=0,\ldots, m$, let $X_i \subset \widehat{\PP}_i$ be the Calabi-Yau variety introduced in the previous section. Recall that $\dim X_i = n(m-1)-1$ and $\rho(X_i) = m$.

Consider the Coxeter system $(W,S)_Q$, with $S = \{s_1,\ldots,s_m\}$, $W = \langle s_i \mid (s_is_j)^{m_{ij}} = 1\rangle$. If $n = 1$, set $m_{ij} = 3$ for $i \neq j$, and if $n \geq 2$ let $m_{ij} = \infty$ for $i\neq j$. The matrix $Q$ is given by 
\begin{equation}\label{eq_matrixQ}
 Q_{ij} = \begin{cases}
1 & \text{if }i=j\\
-n/2 & \text{if }i\neq j.
\end{cases} 
\end{equation}
Under this setup we have that for $n \geq 2$, $W \cong \ZZ/2\ZZ \ast \cdots \ast \ZZ/2\ZZ$.

The matrices $t_j$ from Subsection \ref{subsection:CG} associated to $(W,S)_Q$ are given by
\[ t_j = \begin{pmatrix}
 \mathbf{e}_1  & \ldots & \mathbf{e}_{j-1} & v_j & \mathbf{e}_{j+1} & \ldots & \mathbf{e}_m 
\end{pmatrix}, \] where $v_j = (n,\ldots ,n,\underset{j\text{-th entry}}{-1},n,\ldots ,n)^T$ and $\mathbf{e}_j$ is the $j$-th standard vector. 
 Also define the matrices $\Per_{\sigma}$, with $\sigma$ an element of the symmetric group $\mathcal{S}_m$, as the matrix that permutes rows (respectively columns) by the permutation $\sigma$. With this definition $\Per_{\sigma}\cdot \Per_{\sigma'} = \Per_{\sigma\sigma'}$.

The following result connects the matrices associated to the flops $\varphi_{i,j}\colon X_i \dashrightarrow X_j$ from Subsection \ref{subsection:Construction} and the matrices $t_k$ obtained from the geometrical representation of this Coxeter system.

\begin{proposition}
For $i<j$, the matrix of $\varphi_{i,j}^*$ with respect to the bases $B_j$ and $B_i$ from Equation \eqref{eq_basisBi}  is \[t_j\cdot \Per_{(i+1,i+2,\ldots,j)^{-1}}.\]
\label{propFlopMatrix}
\end{proposition}

\begin{proof}
Recall the diagram for $\varphi_{i,j}$ from Subsection \ref{subsection:Construction}:
\[
\begin{xy}
\xymatrix{
	X_i  \ar@{-->}[rr]^{\varphi_{i,j}} \ar[dr]_{\pi^i_{i,j}}& & X_j \ar[dl]^{\pi^j_{i,j}}
	\\
	& W_{i,j} &	
}
\end{xy}
\]

From the diagram and the definition of the map $\varphi_{i,j}$ we see directly that 
for $k \not\in \{i,j\}$ we have
\[ \varphi_{i,j}^*(H^j_k) = H^i_k \]
and it remains to determine $\varphi_{i,j}^*(H^j_i)$. 
We consider the push-forwards of $H^i_j$ and $H^j_i$ to $W_{i,j}$. By the definition of the variety $W_{i,j}$, ${(\pi^i_{i,j})}_*H^i_j$ corresponds to the maximal minors of an $n \times (n+1)$ submatrix of $A_{i,j}$ whose rows are  general linear combinations of the rows of $A_{i,j}$, whereas ${(\pi^j_{i,j})}_*H^j_i$ corresponds to the maximal minors of an $(n+1) \times n$ submatrix of $A_{i,j}$ whose columns are  general linear combinations of the columns of $A_{i,j}$. Thus, their sum is a complete intersection of $W_{i,j}$ with a hypersurface of bidegree $(n,\ldots,n)$, corresponding to the determinant of a $n\times n$ submatrix of $A_{i,j}$, that is
\[{(\pi^j_{i,j})}_*H^j_i + {(\pi^i_{i,j})}_*H^i_j = n \left. \left( \sum_{\substack{k=0 \\ k\neq i,j}}^m h_k \right)\right|_{W_{i,j}},\]
where $h_k$ corresponds to the pullback of the hyperplane class of $\PP^n_k$ to $\widehat{\PP}_{i,j}$.
Consequently, on $X_i$ we obtain 
\[ \varphi_{i,j}^*(H^j_i)  = -H^i_j + \sum_{\substack{k=0 \\ k\neq i,j}}^m nH^i_k.\]

Recall from Section \ref{subsection:Construction} that the set $B_i = \{H^i_0,\ldots,H^i_{i-1},H^i_{i+1}, \ldots, H^i_{m}\}$ is a basis of $N^1(X_i)_\RR$. Because $i<j$, the element $H_i^j$ is in the $(i+1)$-th position of the basis $B_j$, while the element $H_j^i$ is in the $j$-th position of the basis $B_i$. After permuting $B_j$ by the cycle $(i+1,i+2,\ldots,j)^{-1}$ the matrix of $\varphi_{i,j}^*$ is given by the columns \[ \begin{pmatrix}
\mathbf{e}_1  & \ldots & \mathbf{e}_{j-1} & v_j & \mathbf{e}_{j+1} & \ldots & \mathbf{e}_m 
\end{pmatrix}, \] where $v_j = (n,\ldots ,n,\underset{j\text{-th entry}}{-1},n,\ldots ,n)^T$. This matrix corresponds to the matrix $t_j$, which proves the statement.
\end{proof}
From an easy computation we obtain the following result:
\begin{lemma}\label{lemmaSwapPerm}
Given $t_i$ as defined above and $\sigma \in \mathcal{S}_m$, we have that \[ \Per_{\sigma}\cdot t_i = t_{\sigma(i)}\cdot \Per_\sigma \]
\end{lemma}

\begin{proposition}
For $i>j$, the matrix of $\varphi_{i,j}^*$ with respect to the bases $B_j$ and $B_i$ is \[t_{j+1}\cdot \Per_{(j+1,j+2,\ldots,i)}.\]
\end{proposition}

\begin{proof}
From the construction we have that $\varphi_{i,j}^* = (\varphi_{j,i}^*)^{-1}$. From Proposition \ref{propFlopMatrix}, $(\varphi_{j,i}^*)^{-1} = \Per_{(j+1,j+2,\ldots,i)}\cdot t_i$ because $t_i^2 = \mathrm{id}$. Applying Lemma \ref{lemmaSwapPerm} we obtain that $\varphi_{i,j}^* = t_{j+1}\cdot \Per_{(j+1,j+2,\ldots,i)}$.
\end{proof}

\begin{proposition}
Let $\psi_{i,j}:X \dashrightarrow X$ be the composition of flops 
\[
\begin{xy}
\xymatrix{
	X \ar@/^2.2pc/@{-->}[rrrrrr]^{\psi_{i,j}} \ar@{-->}[rr]^{\varphi_{0,i}} \ar[dr]_{\pi^0_{0,i}}& & X_i \ar[dl]^{\pi^i_{0,i}}
	\ar@{-->}[rr]^{\varphi_{i,j}} \ar[dr]& & X_j \ar[dl]
	\ar@{-->}[rr]^{\varphi_{j,0}} \ar[dr]& & X \ar[dl]
	\\
	& W_{0,i} && W_{i,j} && W_{j,0}
}
\end{xy}.
\]

Then $\psi_{i,j}^* = t_it_jt_i\cdot \Per_{(i,j)}$.
\label{propBirAutMatrix}
\end{proposition}

\begin{proof}
Assume that $i < j$. Then by Proposition \ref{propFlopMatrix} 
\begin{align*}
\psi_{i,j}^* 	&= \varphi_{0,i}^*\cdot \varphi_{i,j}^* \cdot (\varphi_{0,j}^*)^{-1}\\
				&= t_i\cdot \Per_{(1,\ldots,i)^{-1}}\cdot t_j\cdot \Per_{(i+1,\ldots,j)^{-1}}\cdot \Per_{(1,\ldots,j)}\cdot t_j.\\
\intertext{Because $i < j$, by Lemma \ref{lemmaSwapPerm} we obatin $\Per_{(1,\ldots,i)^{-1}}\cdot t_j = t_j\cdot \Per_{(1,\ldots,i)^{-1}}$. Then}\\
\psi_{i,j}^* 	&= t_i\cdot t_j \cdot \Per_{(1,\ldots,i)^{-1}} \cdot \Per_{(i+1,\ldots,j)^{-1}}\cdot \Per_{(1,\ldots,j)}\cdot t_j. \\
\end{align*} 
We need to compute $(1,\ldots,i)^{-1}(i+1,\ldots,j)^{-1}(1,\ldots,j)$. If $k < i$, then 
\begin{align*}
(1,\ldots,i)^{-1}(i+1,\ldots,j)^{-1}(1,\ldots,j)[ k] 	&= (1,\ldots,i)^{-1}(i+1,\ldots,j)^{-1}[k+1]\\
														&= (1,\ldots,i)^{-1}[k+1]\\
														&= k.
\end{align*} 
If $i < k < j$, then 
\begin{align*}
(1,\ldots,i)^{-1}(i+1,\ldots,j)^{-1}(1,\ldots,j)[k] 	&= (1,\ldots,i)^{-1}(i+1,\ldots,j)^{-1}[k+1]\\
														&= (1,\ldots,i)^{-1}[k]\\
														&= k.
\end{align*} 
If $j<k$, then
\begin{align*}
(1,\ldots,i)^{-1}(i+1,\ldots,j)^{-1}(1,\ldots,j)[k] 	&= (1,\ldots,i)^{-1}(i+1,\ldots,j)^{-1}[k]\\
														&= (1,\ldots,i)^{-1}[k]\\
														&= k.
\end{align*} 
Finally, 
\begin{align*}
(1,\ldots,i)^{-1}(i+1,\ldots,j)^{-1}(1,\ldots,j)[i] 	&= (1,\ldots,i)^{-1}(i+1,\ldots,j)^{-1}[i+1]\\
														&= (1,\ldots,i)^{-1}[j]\\
														&= j,\\
(1,\ldots,i)^{-1}(i+1,\ldots,j)^{-1}(1,\ldots,j)[j] 	&= (1,\ldots,i)^{-1}(i+1,\ldots,j)^{-1}[1]\\
														&= (1,\ldots,i)^{-1}[1]\\
														&= i.										
\end{align*}
Therefore, $(1,\ldots,i)^{-1}(i+1,\ldots,j)^{-1}(1,\ldots,j) = (i,j)$.

Back to our computation, we get 
\begin{align*}
\psi_{i,j}^* 	&= t_i\cdot t_j \cdot \Per_{(1,\ldots,i)^{-1}} \cdot \Per_{(i+1,\ldots,j)^{-1}}\cdot \Per_{(1,\ldots,j)}\cdot t_j \\
				&= t_i\cdot t_j \cdot \Per_{(i,j)}\cdot t_j\\
				&= t_it_jt_i\cdot \Per_{(i,j)}.
\end{align*}

For $i>j$, notice that $\psi_{i,j}^* = (\psi_{j,i}^*)^{-1}$, and so $\psi_{i,j}^* = \Per_{(i,j)}\cdot t_jt_it_j = t_it_jt_i\cdot \Per_{(i,j)}$.
\end{proof}

\begin{proposition}\label{prop_diag}
For $i\neq j$ and $1 \leq i,j \leq m$, and consider the birational automorphism $\psi_{i,j}$. If $n = 1$, then $\psi_{i,j}^*$ has order 2, and if $n \geq 2$, $\psi_{i,j}^*$ has infinite order. In the case of $n = 2$ we get that 1 is the only eigenvalue, and that $\psi_{i,j}^*$ is not diagonalizable. When $n \geq 3$, $\psi_{i,j}^*$ is diagonalizable and the eigenvalues are $\lambda$, $1$ and $\lambda^{-1}$ for some $\lambda > 1$. Even more, the eigenvalue $\lambda$ has multiplicity one.
\end{proposition}

\begin{proof}
Combining the results from Lemma \ref{lemmaSwapPerm} and Proposition \ref{propBirAutMatrix} we get that 
\[
\psi_{i,j}^* = t_it_jt_i\cdot \Per_{(i,j)} = \Per_{(i,j)} \cdot t_jt_it_j.
\]

Hence $(\psi_{i,j}^*)^2 = t_it_jt_it_jt_it_j = (t_it_j)^3$. Recall that for $n = 1$, $m_{ij} = 3$, so $(t_it_j)^3 = 1$. For the case $n \geq 2$ it is enough to compute the eigenvalues of $t_it_j$. To do so, we follow the same computation done in \cite[Proposition 4.11]{Y21}.

We write the matrices $t_i$ and $t_j$ in terms of their columns
\[t_i = \begin{pmatrix}
\mathbf{e}_1  & \ldots & \mathbf{e}_{i-1} & v_i & \mathbf{e}_{i+1} & \ldots & \mathbf{e}_m 
\end{pmatrix} \]
and
\[t_j = \begin{pmatrix}
\mathbf{e}_1  & \ldots & \mathbf{e}_{j-1} & v_j & \mathbf{e}_{j+1} & \ldots & \mathbf{e}_m 
\end{pmatrix}, \] where $v_k = (n,\ldots ,n,\underset{k\text{-th entry}}{-1},n,\ldots ,n)^T$.
By computing the product $t_it_j$ we get that the characteristic polynomial of it is \[ \mathrm{cp}_{t_it_j}(x) = (x-1)^{m-2}(x^2 - (n^2-2)x + 1). \]

If $n = 2$, then $\mathrm{cp}_{t_it_j}(x) = (x-1)^m$, but because $t_it_j$ is not the identity matrix, we have that $t_it_j$ has infinite order and it is not diagonalizable.

If $n \geq 3$, then $x^2 - (n^2-2)x + 1$ has two different real roots $\lambda_1$, $\lambda_2$. Because the sum of them is a positive number and $\lambda_1\lambda_2 = 1$, then one and only one of them must be greater than 1.
\end{proof}
Next we describe the automorphism and birational automorphism group of $X$. 
\begin{proposition}\label{prop_AutActTrivial}
	For a general $X$, the automorphisms of $X$ act trivially on $\MovE(X)$. 
\end{proposition}
\begin{proof}
	We proceed similarly as in \cite[Theorem 3.3]{CO15}.
	First we remark that for a general choice of $n+1$ $(1,\ldots,1)$-forms in $\widehat{\PP}_0$, the corresponding variety $X$ is a smooth Calabi-Yau variety. Moreover, we can assume that  any $n$ of the $n+1$ forms define  also a smooth variety which is then a Fano manifold $Y$.  Then $X$ is a smooth anticanonical section of $Y$ and thus,  $\Aut(X)$ is finite by \cite[Theorem 3.1]{CO15}. 
	
	The group $\Aut(X)$ preserves the ample cone of $X$ and hence fixes the set $\{H^0_1,\ldots,H^0_m\}$. 
	Similarly as in the proof of \cite[Theorem 3.3]{CO15} we have 
	\[
	H^0(\widehat{\PP}_0,H_i) = H^0(X, H^0_i)
	\]
	and 
	$\Aut(X)$ is a finite subgroup of 
	\[
	\Aut(\widehat{\PP}_0) = PGL_{n+1}(\CC)^m \rtimes \mathcal{S}_m. 
	\]
	Let $G$ be the image of 
	\[
	\Aut(X) \rightarrow PGL_{n+1}(\CC)^m \rtimes \mathcal{S}_m \rightarrow \mathcal{S}_m.
	\]
	For an element $\id \neq g \in G$, there exists a lift to an automorphism of $\widehat{\PP}_0$ which restricts to an automorphism of $X$. By definition, the automorphism permutes some of the $\PP^n$s. Thus, the ideal generated by the $(1,\ldots,1)$-forms defining $X$ must contain a form which is invariant under this permutation  which is impossible for a general choice of the defining equations. Consequently, $G = \{\id\}$. 
	
	Thus, $\Aut(X)$ can be identified with a finite subgroup of $PGL_{n+1}(\CC)^m$ and therefore it acts trivially on $\MovE(X)$. 
\end{proof}

When $n$ and $m$ are sufficiently large, we obtain the following stronger result:

\begin{proposition}\label{prop_AutTrivial}
	Assume that $n$ and $m$ are greater or equal than $3$. Then for a very general $X$, the group $\Aut(X)$ is trivial.
\end{proposition}

\begin{proof}
	
	From the proof of Proposition \ref{prop_AutActTrivial} we have that $\Aut(X)$ can be identified with a finite subgroup of $PGL_{n+1}(\CC)^m$. Let $g \in PGL_{n+1}(\CC)^m$ be a non-trivial element of finite order which induces an automorphism of $X$. Up to conjugacy, the co-action of $g$ is given by 
	\[
	g^*(x^i_j) = c_{i,j} x^i_j
	\]
	where  each $c_{i,j}$ for $i=1,\ldots,m$ and $j=0,\ldots,n$ is a root of unity.
	Let $f_0,\ldots,f_n$ be the $(1,\ldots,1)$-forms defining $X \subset \widehat{\PP}_0$. As $X$ is invariant under $g$ and all defining equations $f_i$ have the same multidegree, there exists a matrix $M = (m_{i,j}) \in GL_{n+1}(\CC)$ of finite order such that 
	\[
	M  \begin{pmatrix}
	g^*f_0  \\ 
	\vdots \\
	g^*f_n 
	\end{pmatrix} = 
	\begin{pmatrix}
	f_0  \\ 
	\vdots \\
	f_n 
	\end{pmatrix}
	\]
	with $g^*f_i := f_i \circ g$. Moreover, we may assume that the generating equations $f_i$ are chosen so that $M$ is a diagonal matrix whose diagonal elements $m_{i,i}$ are roots of unity. 
	Now, we define the $(1,\ldots,1)$-form  $F$  in $\PP$ as before 
	\[
	F = \sum_{j=0}^n x^0_j f_j.
	\]
 
	Let $\tilde{g}$ be an element in $PGL_{n+1}(\CC)^{m+1} = PGL_{n+1}(\CC) \times  PGL_{n+1}(\CC)^{m} $ which acts via $M^T$ on $\PP^n_0$ and via $g$ on $\widehat{\PP}_0$. Then, by the definition of $M$ and $F$, we have
	\begin{equation}\label{eq_Faut}
	F \circ \tilde{g} = \sum_{j=0}^n (   m_{j,j}x^0_j) g^\ast f_j
	= \sum_{j=0}^n  x^0_j (m_{j,j} g^\ast f_j)= \sum_{j=0}^n x^0_j f_j = F.
	\end{equation}
	A monomial 
	\[ 	x^0_{r_0}x^1_{r_1} x^2_{r_2}\cdots x^m_{r_m}\]  of
	$F$ is mapped to 
	\[ m_{r_0,r_0}	c_{1,r_1}\cdots c_{m,r_m} \cdot x^0_{r_0}
	x^1_{r_1} x^2_{r_2}\cdots x^m_{r_m}
	\]
	where $0 \leq r_i \leq n$ for all $i=0,\ldots,m$. If we set $c_{0,r_0}= m_{r_0,r_0}$, then comparing coefficients of $F$ and $F \circ \tilde{g}$ in \eqref{eq_Faut} implies that 
	\begin{equation}\label{productroots}
	c_{0,r_0}c_{1,r_1}\cdots c_{m,r_m} =1.
	\end{equation}
	for any tuple $(r_0,\ldots,r_m)$ describing a monomial of $F$. 
	As $g$ is not the identity, we assume without loss of generality that $g^*$ acts non-trivially on the last component of $\PP$. 
	If for a fixed choice $(r_0,\ldots,r_{m-1})$, every monomial with index $(r_0,\ldots,r_{m-1},j)$ for $j=0,\ldots,n$ occurs in $F$ with a non-zero coefficient, then \eqref{productroots} implies that 
	\[
	c_{m,0} = c_{m,1} = \ldots = c_{m,n}
	\]
	which is a contradiction. 
	Thus, if  $F$ admits a non-trivial automorphism, the number of monomial factors of $F$ is at most $n(n+1)^m$. Thus, the subset of $(1,\ldots,1)$- forms $F$  which admit a non-trivial automorphism is a countable union of varieties of dimension at most
	
	\[
	n\cdot(n+1)^m -1 + \dim PGL_{n+1}(\CC)^{m+1}
	= n\cdot(n+1)^m -1  + (m+1)\cdot((n+1)^2 - 1).
	\]
	As the space of all possible $(1,\ldots,1)$-forms in $\PP$ has dimension $(n+1)^{m+1} -1 $, the codimension of this subvariety is
	\[
	(n+1)^{m+1}  -1 - [n\cdot(n+1)^m -1  + (m+1)\cdot((n+1)^2 - 1) ] = 
	(n+1)^m - (m+1)[(n+1)^2-1].
	\]
	If $n \geq 3$ and $m\geq 3$, the expression above is a strictly increasing function in either $n$ or $m$. Given that for $n = m = 3$ the expression is equal to 4, we have that the codimension is $\geq 4$. Consequently, removing a countable union of subsets of codimension $\geq 4$, the very general $F$ (and hence $X$) has a trivial automorphism group. 
\end{proof}

\begin{remark}
	The previous argument fails when at least one of $m$ and $n$ is strictly smaller than 3. Example \ref{examp_trivialActing} provides an example in the case of $n=4$ and $m=2$ where the group $\Aut(X)$ is not trivial.
\end{remark}

\begin{proposition}\label{prop_BirFree} For a general X, we have
$\Bir(X) = \Aut(X) \cdot \langle \psi_{i,j} \mid 1 \leq i < j \leq m \rangle$. 
\end{proposition}

\begin{proof}
Let $\alpha \in \Bir(X)$. By {\cite[Theorem 1]{Kaw08}}
every birational map between Calabi-Yau manifolds can be decomposed into a sequence of flops, up to an automorphism. 
By \cite[Corollary 7.6]{Wang21} any flop of a minimal model $X_i$ is among the maps $\varphi_{i,j}\colon X_i \dashrightarrow X_j$ described in Subsection \ref{subsection:Construction}. Hence we can write \[ \alpha = \varphi_{i_s,0}\circ \varphi_{i_{s-1},i_{s}}\circ \cdots \circ \varphi_{i_0,i_1} \circ \varphi_{0,i_0}  \] and by introducing elements of the form $\varphi_{0,i}\circ \varphi_{i,0}$ we obtain that \[ \alpha = \psi_{i_{s-1},i_s}\circ \ldots \circ \psi_{i_1,i_2}\circ \psi_{i_0,i_1}. \] Therefore, up to an automorphism of $X$ \[ \Bir(X) = \langle \psi_{i,j} \mid 1 \leq i < j \leq m \rangle. \qedhere \] 
\end{proof}

From Proposition \ref{prop_diag} we know that when $n \geq 2$, the generators $\psi_{i,j}^*$ have infinite order. We prove, even more, that the group generated by the $\psi_{i,j}^*$ is free of rank ${m \choose 2}$.

\begin{theorem}\label{thm_BirFree}
Assume $n \geq 2$. For a general $X$, we have 
\[
\Bir(X)   \cong \underbrace{\ \ZZ \ast \cdots \ast \ZZ\ }_{{m\choose 2} \text{ terms}}\] up to the finite kernel $\Aut(X)$.
\end{theorem}

We need to prove that there are no non-trivial relations between the generators $\psi_{i,j}$ of $\Bir(X)$. Note that for the following arguments we also need to assume that $n \geq 2$.  
\begin{lemma}
Let $\ww = \psi^*_{i_1,j_1}\psi^*_{i_2,j_2}\cdots \psi^*_{i_s,j_s}$ be a freely reduced word, that means that there are no inverses next to each other. Write $\ww = t_{k_1}t_{k_2}\cdots t_{k_r}\Per_\sigma$ in terms of the elements $t_i$ associated to the Coxeter system $(W,S)_Q$, where $t_{k_1}t_{k_2}\cdots t_{k_r}$  is also the freely reduced word. Then $k_1 = i_1$ and $k_2 = j_1$.
\label{lemma_ReducedWord}
\end{lemma}

\begin{proof}
Let $\ww = \psi^*_{i_1,j_1}\psi^*_{i_2,j_2}\cdots \psi^*_{i_s,j_s} = t_{k_1}t_{k_2}\cdots t_{k_r}\cdot\Per_\sigma$ be freely reduced. Define $\ell_\psi(\ww):= s$ and $\ell_t(\ww) := r$ the length of the word $\ww$ in terms of the generators $\psi_{i,j}^*$ and in terms of the elements $t_i$. Notice that being freely reduced means that for all $1 \leq l < s$, we cannot have simultaneously $i_{l+1} = j_l$ and $j_{l+1} = i_l$.

We will prove the lemma by induction on length $\ell_\psi(\ww)$. Let $\psi^*_{i_1,j_1} = t_{i_1}t_{j_1}t_{i_1}\Per_{\sigma_1}$ and $\psi^*_{i_2,j_2} = t_{i_2}t_{j_2}t_{i_2}\Per_{\sigma_2}$ be two generators, where $\sigma_1 = (i_1,j_1)$ and $\sigma_2 = (i_2,j_2)$. Then 
\[  
\psi^*_{i_1,j_1}\psi^*_{i_2,j_2} = t_{i_1}t_{j_1}t_{i_1}t_{\sigma_1(i_2)}t_{\sigma_1(j_2)}t_{\sigma_1(i_2)}\Per_{\sigma_1\sigma_2}.
\]

We notice that if $i_2 \neq j_1$, then $\ell_t(\psi^*_{i_1,j_1}\psi^*_{i_2,j_2}) = 6$, which means in particular that the first three factors of $\psi^*_{i_1,j_1}\psi^*_{i_2,j_2}$ when written in terms of the $t_i$'s are $t_{i_1}t_{j_1}t_{i_1}$. In the case when $i_2 = j_1$ we have that 
\[
\psi^*_{i_1,j_1}\psi^*_{i_2,j_2} =  t_{i_1}t_{j_1}t_{j_2}t_{i_1}\cdot \Per_{\sigma_1\sigma_2}.
\] In both cases the first two factors are $t_{i_1}t_{j_1}$, so the statement holds. 

Assume that the statement is true for all words of length $s$ and suppose that $\ww$ has length $s+1$. Then 
\[ 
\ww = \psi_{i_1,j_1}\cdot \ww',
\] with $\ww' = \psi^*_{i_2,j_2}\cdots \psi^*_{i_s,j_s} = t_{i_2}t_{j_2}t_{k_3}\ldots t_{k_r}\cdot \Per_{\sigma}$. Then 
\[
\ww = t_{i_1}t_{j_1}t_{i_1}\Per_{(i_1,j_1)}t_{i_2}t_{j_2}t_{k_3}\ldots t_{k_r}\cdot \Per_{\sigma}.
\] Because we don't have simultaneously that $i_2 = j_1$ and $j_2 = i_1$, then
\[ 
\ww = \begin{cases}
t_{i_1}t_{j_1}t_{i_1}\cdots & \text{if } i_2 \neq j_1\\
\\
t_{i_1}t_{j_1}t_{j_2}\cdots & \text{if } i_2 = j_1
\end{cases}
\] and in both cases we have the required result.
\end{proof}
\begin{proof}[Proof of Theorem \ref{thm_BirFree}]
Let us assume that
\[ 
\psi^*_{i_1,j_1}\psi^*_{i_2,j_2}\cdots \psi^*_{i_s,j_s} = \psi^*_{i'_1,j'_1}\psi^*_{i'_2,j'_2}\cdots \psi^*_{i'_r,j'_r}
\] are two reduced words, and let $t_{k_1}t_{k_2}\cdots t_{k_{r'}}\Per_\sigma$ and $t_{k'_1}t_{k'_2}\cdots t_{k'_{s'}}\Per_{\sigma'}$ their associated reduced words in terms of the elements $t_i$ of the Coxeter system $(W,S)_Q$. Then
\[ 
\Per_{\sigma'\sigma^{-1}} = t_{k'_{s'}}\cdots t_{k'_2}t_{k'_1}t_{k_1}t_{k_2}\cdots t_{k_{r'}}
\] 

and because the group generated by the $t_i$'s is isomorphic to $\ZZ/2\ZZ * \ldots * \ZZ/2\ZZ$, we have that $\Per_{\sigma\sigma'} = \mathrm{id}$, $r' = s'$, and $k_l = k'_l$ for all $1 \leq l \leq s'$.

Let $\ww = \psi^*_{i_1,j_1}\psi^*_{i_2,j_2}\cdots \psi^*_{i_s,j_s}$ be a non-trivial freely reduced word such that $\ww = \mathrm{id}$. By Lemma \ref{lemma_ReducedWord} we have that $\ell_t(\ww) \geq 2$, but $\ell_t(\mathrm{id}) = 0$ which is a contradiction. Therefore there is no non-trivial relation between the $\psi_{i,j}^*$ , hence, modulo automorphisms, we have  
\[ 
\langle \psi_{i,j} \mid 1 \leq i < j \leq m \rangle \cong \bigast_{1\,\leq\, i< j\, \leq\, m} \langle \psi^*_{i,j} \rangle  \cong \underbrace{\ \ZZ \ast \cdots \ast \ZZ\ }_{{m\choose 2} \text{ terms}}.
\]
\end{proof}

In \cite[Theorem 1.3]{Wang21} it is shown that there is a fundamental domain $\Pi$ for the action of $\Bir(X)$ on $\MovE(X)$. Whereas the proof is non-constructive, we can now describe the fundamental domain $\Pi$ explicitly in the general case.

\begin{proposition}\label{prop_FundamentalDom} Let $X_0=X$ be general, and let 
	\[ \Pi = \Nef^e(X_0) \cup \bigcup_{i=1,\ldots,m} \varphi_{0,i}^*\Nef^e(X_i) \subset \MovE(X).\]
	Then $\Pi$ is a fundamental domain for the action of $\Bir(X)$ on $\MovE(X_0)$, that means
	\[ \MovE(X) = \bigcup_{g\in \Bir(X)} g^*\Pi
	\]
	and $\Int \Pi \cap \Int g^*\Pi = \emptyset$ unless $g^* = \id$. 
\end{proposition}
\begin{proof} 
	By \cite[Proposition 4.7 and Corollary 7.9]{Wang21} we have 
	\[ \MovE(X) = \bigcup_{(X',\alpha)} \alpha^*\Nef^e(X')
	\]
	where the union runs over all minimal models of $X$ and markings $\alpha\colon X \dashrightarrow X'$. Recall that $X$ has exactly $m+1$ minimal models $X_0 = X, X_1,
	\ldots, X_m$ up to isomorphism. Let $E \in \MovE(X)$. Then there exist a minimal model $X_j$, a birational map $ \alpha_j\colon X \dashrightarrow X_j$ which is a sequence of flops and a divisor $E_j \in \Nef^e(X_j)$ such that $E = \alpha_j^*(E_j)$.  By \cite[Corollary 7.7]{Wang21} any flop of a $X_k$ is among the $\varphi_{k,l}$ constructed in Subsection \ref{subsection:Construction}. Thus, up to isomorphism, we have 
	\[
	\alpha_j = \varphi_{i_s,j}\circ \varphi_{i_{s-1},i_{s}}\circ \cdots \circ \varphi_{i_0,i_1} \circ \varphi_{0,i_0}.
	\]
	Now, introducing elements of the form $\varphi_{0,i_k}\circ \varphi_{i_k,0}$ as in the proof of  Proposition \ref{prop_BirFree}, we have
 \[ \alpha_j = \varphi_{0,j}\circ \underbrace{\psi_{i_{s},j}\circ \ldots \circ \psi_{i_1,i_2}\circ \psi_{i_0,i_1}}_{ =: g \in \Bir(X)}. \] 
 Consequently, we have 
 \[E = g^*(\varphi_{0,j}^*E_j) \in g^*(\varphi_{0,j}^*\Nef^e(X_j)) \subset g^* \Pi. 
\]
Finally, let us assume that there is a $g \in \Bir(X)$ such that $\Int  \Pi   \cap  \Int g^*\Pi \neq \emptyset$. By the definition of $\Pi$ this implies that there exists a divisor $E_k \in \Amp(X_k)$ such that (after composing with some $\varphi_{0,k}$ if necessary) $g^*E_k$ is ample on $X_k$ or some other minimal model. This implies that $g$ (composed with some $\varphi_{0,k}$) is an automorphism, and hence,  $g^* = \id$ by Proposition \ref{prop_AutActTrivial}.
\end{proof}

To describe the movable cone we will use the connection between the generators of $\Bir(X)$ and the Coxeter system. To see this, let $D$ be, as in Section \ref{subsection:CG}, the fundamental domain of the action of the dual representation of $W$, and let $T$ the Tits cone of $W$. We can identify $V^*$ with $N^1(X)_\RR$ as follows: Let $\{\alpha_i\}_{i=1}^m$ be the basis of $V$, and consider $\{\mathbf{\beta}_i\}_{i=1}^m$ as the vectors in $V$ such that $\mathcal{Q}(\alpha_i,\beta_j) = \delta_{i,j}$, where $\delta_{i,j}$ is the Kronecker Delta. Then the fundamental domain $D$ is the cone $\mathrm{Cone}(\{\beta_i\}_{i=1}^m)$ generated by the $\beta_i$'s and the generators of the dual representation of $W$ correspond to the elements $t_i$. Finally, identify the fundamental domain $D$ of $W$ with $\Nef(X_0)$ via the map $\beta_i \mapsto H_i^0$. Note that as cones we have $\varphi_{0,i}^* \Nef(X_0) = t_i\Per_{(1,\ldots,i)^{-1}}. \Nef(X_0) = t_i. \Nef(X_0)$. 

\begin{proposition}
Identify $D$ with $\Nef(X)$ as above. Then the Tits cone $T$ is isomorphic to $\MovE(X)$.
\label{prop_titsiso}
\end{proposition}

\begin{proof}
From Proposition \ref{prop_FundamentalDom} we have that $\Pi = \Nef^e(X_0) \cup \bigcup_{i=1,\ldots,m} \varphi_{0,i}^*\Nef^e(X_i)$ is a fundamental domain for the action of $\Bir(X)$ on $\MovE(X)$. Also, notice that if $\ww = t_{i_1}t_{i_2}\cdots t_{i_s}$, then $\ww.D = \ww\cdot \Per_\sigma.D$. In particular, $\Pi = \NefE(X_0) \cup \bigcup_{i=1}^m t_i.\NefE(X_0)$, so $\MovE(X) \subseteq T$.

To prove the opposite inclusion, we need to show that given a freely reduced word $\ww = t_{i_1}t_{i_2}\cdots t_{i_s}$, there exists a word $\ww' = \psi_{i'_1,j'_1}^*\psi_{i'_2,j'_2}^*\cdots \psi_{i'_r,j'_r}^*$ such that the chamber $\ww.D$ of the Tits cone is included in the chamber $\ww'.\Pi$ of the movable cone. We construct the word $\ww'$ by induction on the length of $\ww$. If $\ww = t_i$, then $\ww.D \subseteq \Pi$, so $\ww' = \mathrm{id}$. If $\ww = t_{i_1}t_{i_2}\cdots t_{i_s}$, with $s \geq 2$, we can write \[ 
\ww =  t_{i_1}t_{i_2}t_{i_3}\cdots t_{i_s} = (t_{i_1}t_{i_2}t_{i_1})t_{i_1}t_{i_3}\cdots t_{i_s}.
\]

Then we set $\ww' = \psi_{i_1,i_2}^*\ww_1'$, where $\ww_1'$ is the word associated to 
\[
\ww_1 =
t_{\sigma(i_1)}t_{\sigma(i_3)}\cdots t_{\sigma(i_s)} = t_{i_2}t_{\sigma(i_3)}\cdots t_{\sigma(i_s)} 
\] with $\sigma = (i_1,i_2)$.  The word $\ww'_1$ exists because the length of $\ww_1$ is less than $s$. Therefore $T \subseteq \MovE(X)$.
\end{proof}

To describe the boundary of $\MovC(X)$ we can use the fact that the underlying Coxeter system $(W,S)_Q$ is Lorentzian.

\begin{proposition}
The Coxeter system $(W,S)_Q$ is Lorentzian.
\label{prop_lorentzian}
\end{proposition}

\begin{proof}
The statement follows from the fact that for $n \geq 1$, the matrix $Q$ from Equation \eqref{eq_matrixQ} has eigenvalues $\lambda_1 = 1+\frac{n}{2} > 0$ with eigenvectors $\mathbf{e}_1 - \mathbf{e}_i$, for $2 \leq i \leq m$; and $\lambda_2 = 1 - \frac{n(m-1)}{2} < 0$, with eigenvector $(1,1,\ldots,1)$.
\end{proof}

Given that $\MovE(X)$ is equal to the Tits cone, and that the Coxeter system is Lorentzian, we can replicate the proof of \cite[Theorem 4.10]{Y21} to compute the boundary of $\MovC(X)$ (see Remark \ref{remark_boundaryN1} for the case of $n = 1$.).

\begin{proposition}[cf. {\cite[Theorem 4.10]{Y21}}]
The boundary of $\MovC(X)$ is the closure of the union of the $W$-orbits of the cones  $\{ a_\lambda v_\lambda + \sum_{i,j\neq k} a_k H^0_k \mid a_\lambda \geq 0, a_k\geq 0\}$, where $v_\lambda$ is:
\begin{itemize}
\item If $n \geq 3$, an eigenvector associated to the unique eigenvalue $\lambda >1$ of $t_it_j$; or
\item If $n = 2$, $v_\lambda = 0$.
\end{itemize}
\label{prop_movBoundary}
\end{proposition}

\begin{proof}
The proof follows exactly as in \cite[Theorem 4.10]{Y21}, with $J = \{1,2,\ldots , m\}$.
\end{proof}

\begin{remark}
Propositions \ref{prop_titsiso} and \ref{prop_lorentzian} still hold when $n = 1$, but in this case the boundary of the Tits cone associated to the Coxeter system $(W,S)_Q$ cannot be described in the same way as for $n \geq 2$ shown in Proposition \ref{prop_movBoundary}.
\label{remark_boundaryN1}
\end{remark}

\begin{remark}
The boundary of the movable cone $\MovC(X)$ can also be described in terms of $\Bir(X)$-orbits. This will be done in the examples of Section \ref{section_ex}.
\end{remark}

\section{Examples}\label{section_ex}

\subsection{General case}\label{subsection_general}
In this subsection, we study the movable cone in concrete examples. We consider general cases with varying $n$ and fixed Picard number $m=3$. 
\begin{remark}
	In this section, we provide pictures of the movable cone and the nef cone  for $m=3$ created with the software Mathematica \cite{Mathematica}. To do so, we project from $\PP N^1(X)_\RR \cong \PP^2_\RR$ to the affine space $V_1 = \{ v \in \RR^3 \mid \varphi(v)= v_0+v_1+v_2 =1\}$ as described in Remark \ref{rem_AffSpace}. 
\end{remark}
\begin{example}[$n=2$ and $m=3$]\label{examp_3foldn2m3}
	Let $X := X_0 \subset \widehat{\PP}_0 = \PP^2_1 \times \PP^2_2 \times \PP^2_3$ be a general complete intersection of 3 forms of multidegree $(1,1,1)$ and let $F$ be the $(1,1,1,1)$-form in $\PP =\PP^2_0 \times \PP^2_1 \times \PP^2_2 \times \PP^2_3$ as defined in Construction \ref{cons_Xi}. Then $X$ is Calabi-Yau variety with 
	\[ \dim X = n(m-1)-1 = 3 \text{ and } \rho(X) = m = 3.\] The Calabi-Yau varieties $X_i \subset \widehat{\PP}_i$ for $i=1,2, 3$ from Construction \ref{cons_Xi} are the only minimal models of $X$ up to isomorphism. By Proposition \ref{prop_BirFree} we have 
	\[
	\Bir(X)  = \Aut(X) \cdot \langle \psi_{i,j} \mid 1 \leq i < j \leq 3 \rangle
	\]
	and
	\[
	\langle \psi_{i,j}^* \mid 1 \leq i < j \leq 3 \rangle \cong \ZZ \ast \ZZ \ast \ZZ
	\]
	 by Theorem \ref{thm_BirFree}.
	Using Proposition \ref{propBirAutMatrix} we compute that
	\[ \psi_{1,2}^* = 
	\begin{pmatrix}
	-{2}&-{3}&{0}\\
	{3}&{4}&{0}\\
	{6}&{12}&1\\
	\end{pmatrix}, 
	\psi_{1,3}^* = \begin{pmatrix}
-2&0&-3\\
6&1&12\\
3&0&4\\
	\end{pmatrix} \text{ and } 
	\psi_{2,3}^* = \begin{pmatrix}
	 1&{6}&{12}\\
	{0}&-{2}&-{3}\\
	{0}&{3}&{4}\\
	\end{pmatrix}
	\]
	with respect to the bases $B_0$ given in \eqref{eq_basisBi}.
	
We have $\Nef(X) = \text{Cone}(H^0_1,H^0_2,H^0_3)$ and the fundamental domain $\Pi$ from Proposition \ref{prop_FundamentalDom}  in $\PP N^1(X)_\RR$ is hexagonal with vertices
\[ \{ H^0_3, \varphi_{0,1}^*H^1_0 , H^0_2, \varphi_{0,3}^*H^3_0, H^0_1,\varphi_{0,2}^*H^2_0 \}. \]

\begin{figure}[hbt]
  \begin{center}
  \hfill
  \begin{minipage}[l]{.45\textwidth}
 \includegraphics[scale=.36]{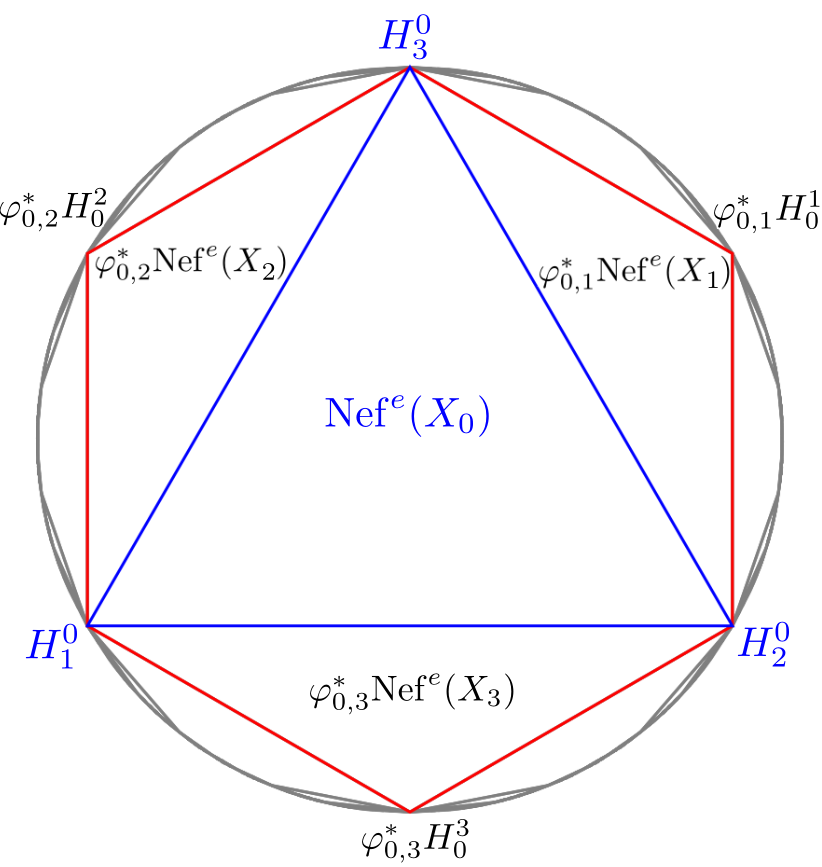}
  \end{minipage} \hfill \hfill \begin{minipage}[r]{.45\textwidth}
  \includegraphics[scale=.34]{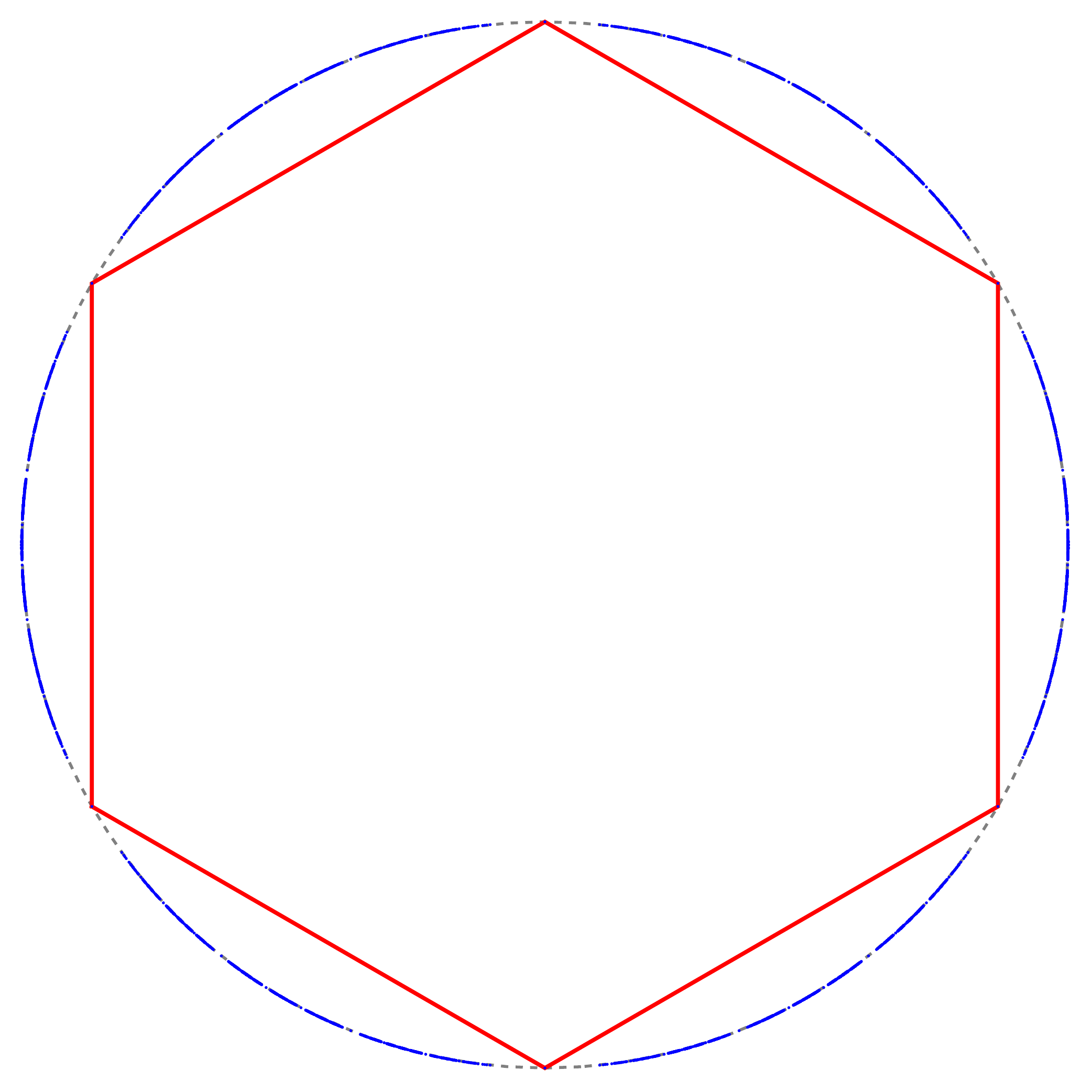}
  \end{minipage}
  \hfill
  \end{center}
  \caption{
  Projective view of the movable cone of $X$ in the case of $m = 3$ and $n = 2$. 
  On the left, the cone obtained as the union of the images of the fundamental domain, in red, under the action of $\Bir(X)$. The nef cone is drawn in blue. On the right, the boundary of the $\MovC(X)$ obtained using Proposition \ref{prop_movBoundary}. In red, the fundamental domain, and in blue the $\Bir(X)$-orbits of the vertices of the fundamental domain.
  }
  \label{figureMovableN2}
\end{figure}
As mentioned in Remark \ref{rem_quadric},  all accumulation points of $\MovC(X)$ lie on the quadratic cone 
\[ \{t_0t_1 + t_0t_2 + t_1t_2=0\}
\] 
in $\PP N^1(X)_\RR \cong \PP^2_\RR$ given by the inverse of the symmetric matrix
\[Q = \begin{pmatrix}
1 & -1 & -1 \\
-1 & 1 & -1 \\
-1 & -1 & 1 \\
\end{pmatrix}\]
 from Equation \eqref{eq_matrixQ} with $n=2$. 

Recall that any codimension one face of $\Nef(X_i)$ corresponds to a contraction from $X_i$ to some $W_{i,j}$ and induces the flop $\varphi_{i,j} \colon X_i \dashrightarrow X_j$. By pulling back to $X$ we can identify each codimension one face of the red hexagon $\Pi$ in Figure \ref{figureMovableN2} with exactly one element in 
\begin{equation}
\label{eq_CodimOne}
\Psi = \{
\psi_{1,2}^*,  \psi_{2,1}^* =(\psi_{1,2}^*)^{-1} , \psi_{1,3}^*, \psi_{3,1}^*, \psi_{2,3}^*, \psi_{3,2}^*
\}.
\end{equation}
Indeed, for each $\psi \in \Psi$, the cone $\psi \Pi$ is again hexagonal and from the definition of the $\psi_{i,j}^*$ we deduce that $\psi \Pi$  has exactly one codimension one face in common with $\Pi$. More, generally, for any word $
\ww = \psi_{i_1,j_1}^*\psi_{i_2,j_2}^*\cdots \psi_{i_r,j_r}^* $, the cones $
\ww \Pi \text { and } \ww' \Pi,
$
for $\ww' = \psi_{i_1,j_1}^*\psi_{i_2,j_2}^*\cdots \psi_{i_{r-1},j_{r-1}}^*$, 
have one common codimension one face  and lie in the segment correponding to $ \psi_{i_1,j_1}^*$.

\end{example}

\begin{example}[$n=3$ and $m=3$ ]\label{examp_3foldn3m3}
	We consider a general $X := X_0 \subset \PP^3_1 \times \PP^3_2 \times \PP^3_3$  which is a smooth complete intersection Calabi-Yau fivefold of Picard number $m=3$. 
	As before, $X$ has 4 minimal models $X,X_1,X_2,X_3$ up to isomorphism
	and 
		\[
	\Bir(X)  = \langle \psi_{i,j} \mid 1 \leq i < j \leq 3 \rangle \cong \ZZ \ast \ZZ \ast \ZZ
	\]
	by Proposition \ref{prop_AutTrivial} and Theorem \ref{thm_BirFree}. We compute that
		\[ \psi_{1,2}^* = 
	\begin{pmatrix}
	 -{3}&-{8}&{0}\\
	{8}&{21}&{0}\\
	{12}&{36}&1\\
	\end{pmatrix}, 
	\psi_{1,3}^* = \begin{pmatrix}
-{3}&{0}&-{8}\\
{12}&1&{36}\\
{8}&{0}&{21}\\
	\end{pmatrix} \text{ and } 
	\psi_{2,3}^* = \begin{pmatrix}
 1&{12}&{36}\\
{0}&-{3}&-{8}\\
{0}&{8}&{21}\\
	\end{pmatrix}
	\]
	with respect to the basis $B_0$ from Equation \eqref{eq_basisBi}.
	
As before we have $\Nef(X) = \text{Cone}(H^0_1,H^0_2,H^0_3)$ and the fundamental domain $\Pi$ from Proposition \ref{prop_FundamentalDom} is hexagonal in $\PP N^1(X)_\RR$  with vertices
\[ \{ H^0_3, \varphi_{0,1}^*H^1_0 , H^0_2, \varphi_{0,3}^*H^3_0, H^0_1,\varphi_{0,2}^*H^2_0 \}. \]
Now, in contrast to the previous example  $\psi_{i,j}^*$ is diagonalizable for any $ i\neq j$ and the boundary of $\MovC(X)$ contains now line segments between the eigenvector of $\psi^*_{i,j}$ associated to the unique eigenvalue $\lambda > 1$ and the vertex corresponding to $H^0_k$ where $1 \leq k \leq 3$ and $k \neq i,j$ as pictured in Figure \ref{figureMovableN3}. 
The inverse of the matrix 
\[Q = \begin{pmatrix}
1 & -\frac{3}{2} & -\frac{3}{2} \\[0.5em]
-\frac{3}{2} & 1 & -\frac{3}{2}\\[0.5em]
-\frac{3}{2} & -\frac{3}{2} & 1 \\
\end{pmatrix}\] from Equation \eqref{eq_matrixQ} with $n=3$  defines a quadratic cone 
\[
\{t_0^2-6t_0t_1+t_1^2-6t_0t_2-6t_1t_2+t_2^2=0\}
\]
in $\PP N^1(X)_\RR \cong \PP^2_\RR$ which contains all these eigenvectors. 
As in the previous example, every edge of the red hexagon in Figure \ref{figureMovableN3} corresponds to one of the elements in \[
\Psi = \{
\psi_{1,2}^*,  \psi_{2,1}^*, \psi_{1,3}^*, \psi_{3,1}^*, \psi_{2,3}^*, \psi_{3,2}^*
\}.
\]
and  $\{(\psi^*_{i_1,j_1})^k \cdot \Pi \}$ gives a sequence of adjacent shrinking cone approaching the eigenvector of $\psi^*_{i_1,j_1}$ associated to the unique eigenvalue $\lambda > 1$. 

\begin{figure}[hbt]
  \begin{center}
  \hfill
  \begin{minipage}[c]{.45\textwidth}
  \includegraphics[scale=.33]{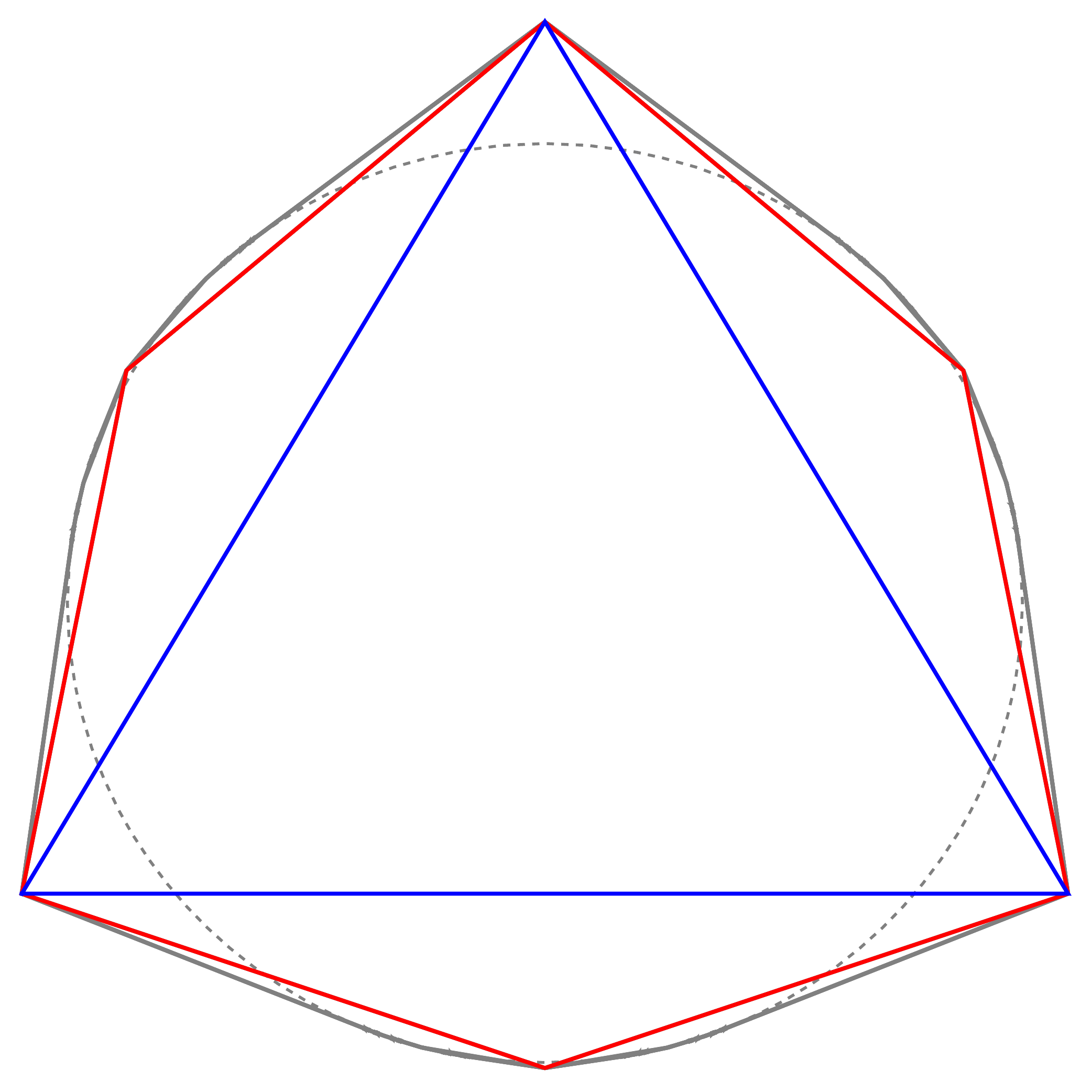}
  \end{minipage} \hfill \hfill \begin{minipage}[c]{.45\textwidth}
  \includegraphics[scale=.33]{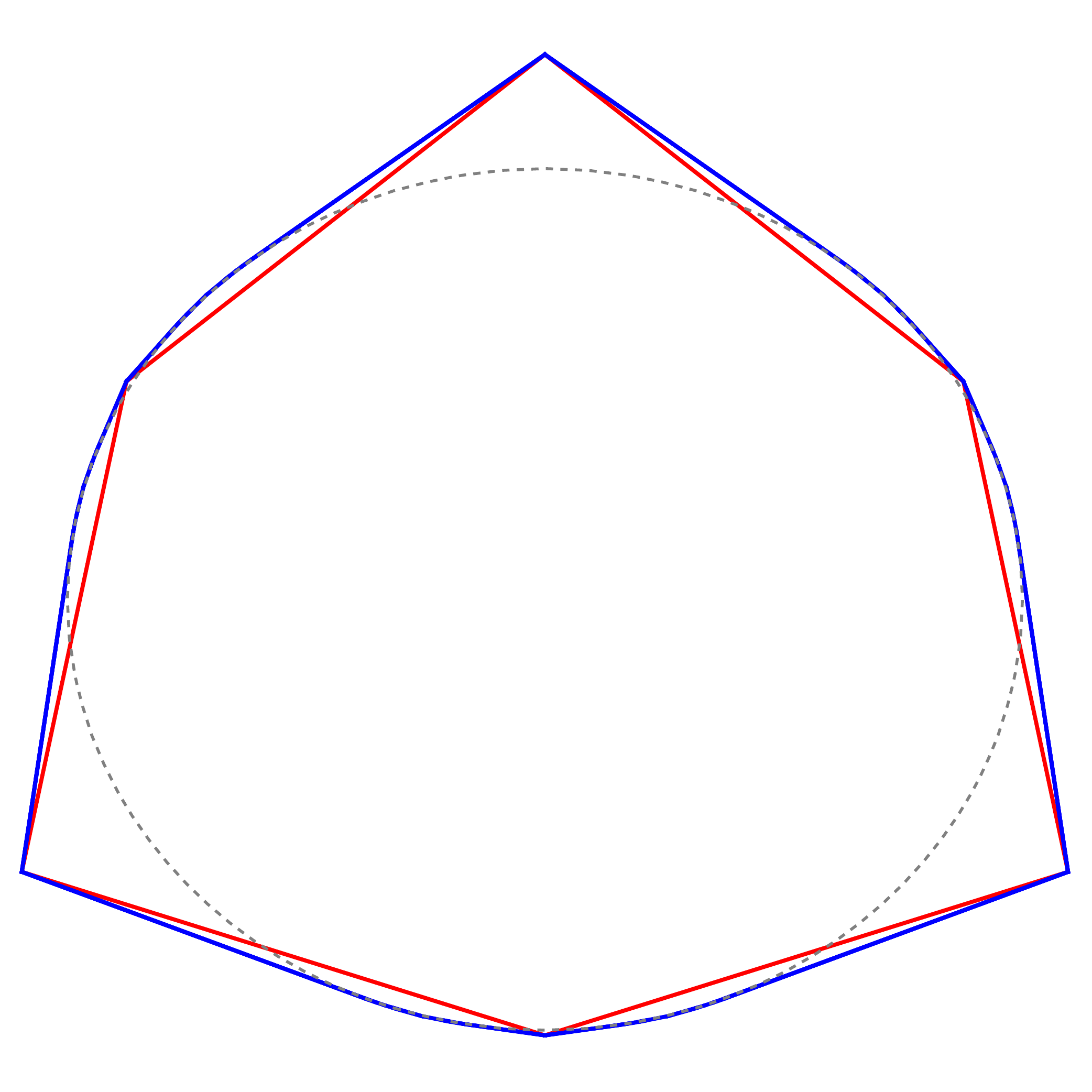}
  \end{minipage}
  \hfill
  \end{center}
  \caption{
  Projective view of the movable cone of $X$ in the case of $m = 3$ and $n = 3$. 
  On the left, the cone obtained as the union of the images of the fundamental domain, in red, under the action of $\Bir(X)$. The nef cone is drawn in blue. On the right, the boundary of the $\MovC(X)$ obtained using Proposition \ref{prop_movBoundary}. In red, the fundamental domain, and in blue the $\Bir(X)$-orbits of the cones described in Proposition \ref{prop_movBoundary}.
  }
  \label{figureMovableN3}
\end{figure}

\end{example}
\subsection{Non-general case with non-trivial automorphism group}\label{subsection_nongeneral}
In this subsection we present some special cases which have a non-trivial automorphism group. 
\begin{example}[$n=2$ and $m=3$]\label{example_nontrivialaut}
	As before, we consider first the complete intersection of 3 forms of multidegree $(1,1,1)$ in $\widehat{\PP}_0 = \PP^2_1 \times \PP^2_2 \times \PP^2_3$.  Let $R_0$ be the the Cox ring of  $\widehat{\PP}_0 $, and let $\textbf{x}^i = [x_0^i,x_1^i,x_2^i]$ be the variables of the individual $\PP^2_i$. Then we define a homomorphism $g$ by
	\begin{align*}
  x_j^1 & \mapsto x_j^2, \\
 x_j^2 & \mapsto x_j^1, \\ 
	x_j^3 & \mapsto x_j^3, 
	\end{align*}
	for $j = 0,1,2$. 
\end{example}
Let $f_0',f_1',f_2'$ be three  $(1,1,1)$-forms in $\widehat{\PP}_0$ and set
 \[ f_i := f_i' + g(f_i').\]
If the $f_i'$ are chosen generically, the variety  $X_0$ defined by $f_0,f_1,f_2$ is a smooth Calabi-Yau threefold in $\widehat{\PP}_0$. Moreover, $X_0$ is invariant under the automorphism of $\widehat{\PP}_0$ induced by $g$. 
Let 
\[ 
F = \sum_{i=0}^{2} x^0_i f_i
\]
be the $(1,1,1,1)$-form in $\PP$. Then we can extend $g$ to the Cox ring of $\PP$ by letting the variables $\textbf{x}^0$ be mapped identically. Thus, $F$ is a homogeneous polynomial which is invariant under interchanging the $\textbf{x}^1$- and $\textbf{x}^2$-variables.  Recall that the $3 \times 3$-matrix $A_{i,j}$ is the matrix of second derivatives of the polynomial $F$ with respect to the homogeneous coordinates $\textbf{x}^i$ and $\textbf{x}^j$. Now, if we proceed with Wang's construction we see that the matrix $A_{1,2}$ defining the variety $W_{1,2}$ is symmetric and that the two projections $\pi^{1}_{1,2}$ and $\pi^{2}_{1,2}$  are no longer small contractions but divisorial. 

Furthermore, the matrix $A_{0,1}$ is the transpose of the matrix $A_{2,0}$ after applying the homomorphism $g$. Thus $W_{0,1}$ and $W_{2,0}$ are isomorphic and from the diagram in Proposition \ref{propBirAutMatrix} we  conclude
that the  two minimal models $X_1$ and $X_2$ are isomorphic with the isomorphism given by
\begin{align*}
\varphi_{1,2}\colon X_1  &\rightarrow X_2, \\
[\textbf{p}^0,\textbf{p}^2,\textbf{p}^3] & \mapsto [\textbf{q}^0,\textbf{q}^1,\textbf{q}^3]:= [\textbf{p}^0,\textbf{p}^2,\textbf{p}^3].
\end{align*}

The birational automorphism $\psi_{1,2}$ from Proposition \ref{propBirAutMatrix} is now an automorphism given by 
\begin{align*}
\psi_{1,2}\colon X_0  &\rightarrow X_0, \\
[\textbf{p}^1,\textbf{p}^2,\textbf{p}^3] & \mapsto [\textbf{p}^2,\textbf{p}^1,\textbf{p}^3];
\end{align*}
and 
	\[
	\psi_{1,2}^* = \begin{pmatrix}
	0 & 1 & 0 \\
	1 & 0 & 0 \\
	0 & 0 & 1
	\end{pmatrix}
	\]
	with respect to the basis $B_0$. The presentation matrices of $\psi_{2,3}^*$ and $\psi_{1,3}^*$ are the same as in Example \ref{examp_3foldn2m3} but due to the isomorphism between $X_1$ and $X_2$ we have
\[
\psi_{2,3} = 	\psi_{1,2} \circ 	\psi_{1,3}\circ 	\psi_{1,2}.
\]
Indeed, by the symmetry of $F$ with respect to $\textbf{x}^1$ and $\textbf{x}^2$ we have an automorphism
\begin{align*}
\alpha \colon X_3  &\rightarrow X_3, \\
[\textbf{p}^0,\textbf{p}^1,\textbf{p}^2] & \mapsto [\textbf{p}^0,\textbf{p}^2,\textbf{p}^1]
\end{align*}
which coincides with $\varphi_{2,3} \circ \varphi_{1,2} \circ \varphi_{3,1}$. 
Then, from the definition of the maps $\psi_{i,j}$ and $\alpha$ we see that
\begin{align*}
\psi_{2,3}   & = \varphi_{3,0} \circ \varphi_{2,3} \circ \varphi_{0,2} \\
& = \varphi_{3,0} \circ \varphi_{2,3} \circ (\varphi_{1,2} \circ \varphi_{0,1} \circ \psi_{1,2} ) \\ 
& = \varphi_{3,0} \circ (\varphi_{2,3} \circ \varphi_{1,2} \circ \varphi_{3,1}) \circ \varphi_{1,3}\circ \varphi_{0,1} \circ \psi_{1,2}  \\ 
& = \varphi_{3,0} \circ \alpha \circ  \varphi_{1,3} \circ \varphi_{0,1} \circ \psi_{1,2} \\ 
& = (\varphi_{3,0} \circ \alpha \circ \varphi_{0,3} )\circ   \psi_{1,3}  \circ \psi_{1,2}  = 
\psi_{1,2} \circ \psi_{1,3}  \circ \psi_{1,2},
\end{align*}
where the last equality follows from the fact that the matrix $A_{0,3} = A_{3,0}^T$ is invariant under interchanging $\textbf{x}^1$ and $\textbf{x}^2$.
Assuming that there are no further automorphisms we have
\[
\Bir(X_0) = \langle \psi_{1,2}, \psi_{1,3} \rangle \cong \ZZ/2\ZZ \ast \ZZ.
\]
The cones $\Nef(X_0)$ and $\varphi_{0,3}^*(\Nef(X_3))$  are preserved under $\psi_{1,2}^*$, whereas $\varphi_{0,2}^*(\Nef(X_2))$ is mapped to the cone $\varphi_{0,1}^*(\Nef(X_1))$. Thus, if $\Pi$ is the convex cone generated by the classes  $H^0_3$, $\varphi_{0,1}^*H^1_0$, $H^0_2$ and $\varphi_{0,3}^*H^3_0$, then 
\[\Pi \cup \psi_{1,2}^* \Pi  = \NefE(X_0) \cup \bigcup_{i \in \{1,2,3\}}  \varphi_{0,i}^* \NefE(X_i)\]
coincides with the hexagonal cone from Example \ref{examp_3foldn2m3}  and we obtain 
 \[\MovE(X_0) = \bigcup_{g\in \Bir(X)} g^*\Pi\]
 exactly as in the proof of Proposition \ref{prop_FundamentalDom}. Consequently, the Kawamata-Morrison cone conjecture is also satisfied in this non-general case.  Note that Lemma 7.6 up to Corollary 7.9 from \cite{Wang21} still hold with the only exception that $\varphi_{1,2}\colon X_1 \rightarrow X_2$ is an isomorphism and the projections $\pi^1_{1,2}$ and $\pi^2_{1,2}$ are divisorial.

\begin{figure}[hbt]
  \begin{center}
  \hfill
  \begin{minipage}[c]{.47\textwidth}
  \includegraphics[scale=.35]{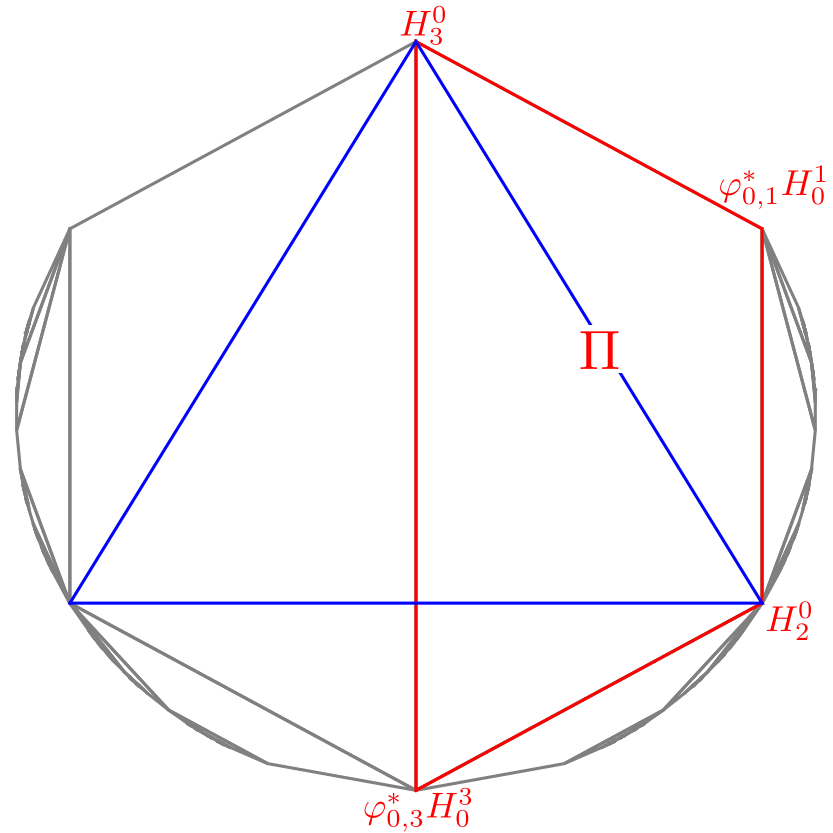}
  \end{minipage} \hfill \hfill \begin{minipage}[c]{.45\textwidth}
  \includegraphics[scale=.33]{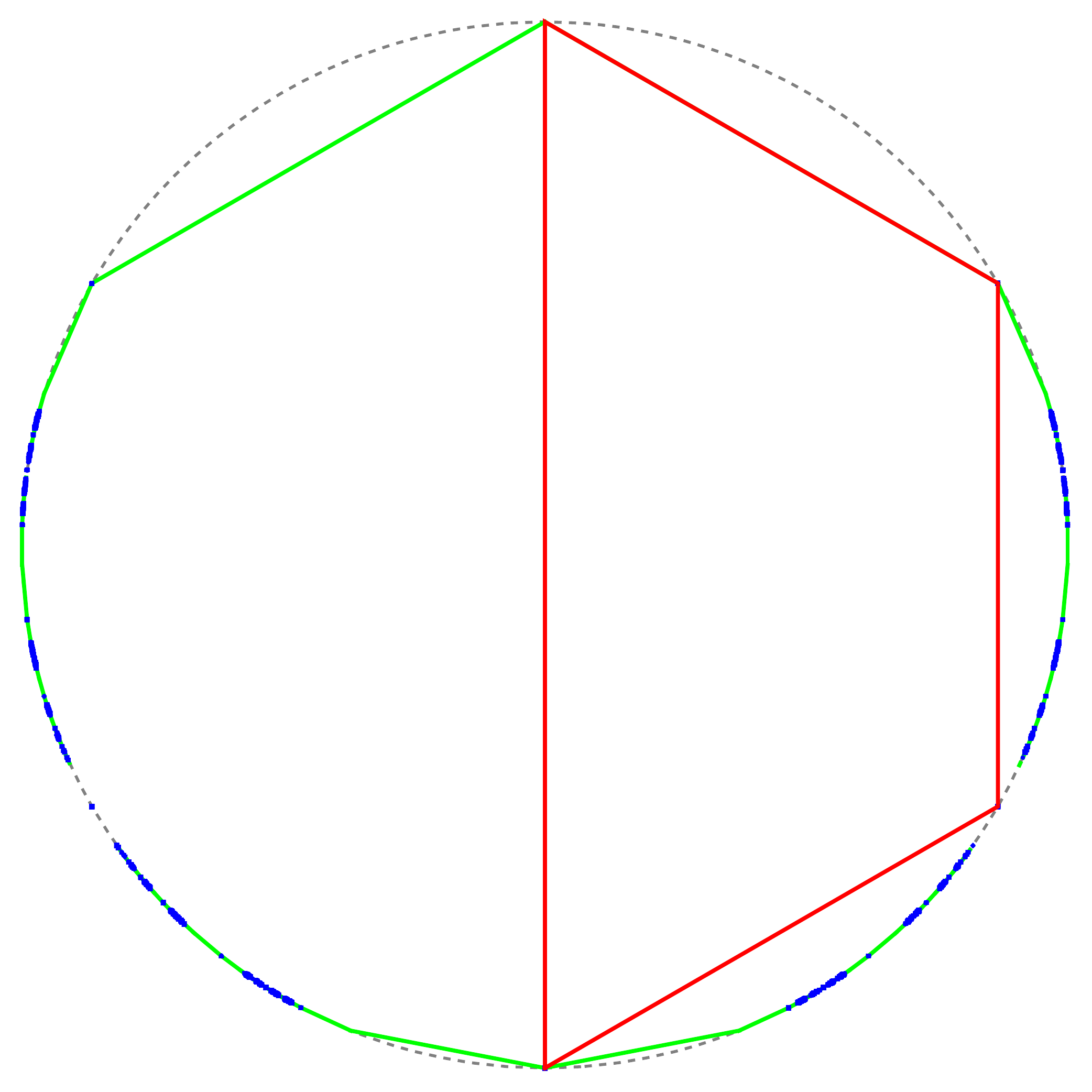}
  \end{minipage}
  \hfill
  \end{center}
  \caption{
  Projective view of the movable cone of $X$ in the case of $m = 3$ and $n = 2$ with a non-trivial automorphism. 
  On the left, the cone obtained as the union of the images of the fundamental domain, in red, under the action of $\Bir(X)$. The nef cone of $X$ is drawn in blue. On the right, the boundary of the $\MovC(X)$. In red, the fundamental domain; in green, the orbit under $\Bir(X)$ of the faces spanned by $\{H^0_3, \varphi_{0,1}^* H^1_0\}$ and $\{H^0_3, \varphi_{0,2}^* H^2_0\}$ corresponding to divisorial contractions of $X_1$ and $X_2$, respectively; and in blue the $\Bir(X)$-orbits of the vertices of the faces of the fundamental domain that correspond to flops.
  }
  \label{figureMovableAut}
\end{figure}
\begin{remark}
	Note that $X_1$ and $X_2$ are isomorphic as minimal models of $X_0$ but not as marked minimal models $\varphi_{0,1}\colon  X_0 \dashrightarrow X_1$ and $\varphi_{0,2}\colon  X_0 \dashrightarrow X_2$. Thus pulling the respective nef cones back to $X_0$, we obtain two different chambers of $\MovC(X)$ whose interiors do not intersect. 
\end{remark}

The question is how the pseudoeffective cone $\EffC(X)$ looks like in this example. A first natural guess is that $\EffC(X)$ is bounded by the quadratic cone
\[ \mathcal{C}_Q := \{t_0t_1 + t_0t_2 + t_1t_2 =0 \} \subset \PP N^1(X)_\RR\] from the  non-symmetric case of Example \ref{examp_3foldn2m3}.  However, using \textit{Macaulay2} (see \cite{macaulay2} and \cite{CYCIM2}), we compute that
\[
h^0(X, D_1) = 1
\]
where $D_1$ is  an effective integral divisor whose numerical class is equivalent to $-2H^0_1 + 2H^0_2 + 6H^0_3$. 
We note that the class of this divisor lies outside of the quadratic cone and is
 the intersection of the two tangent lines of the quadratic cone in $\PP N^1(X)_\RR$ to $H^0_3$ and $\varphi_{0,1}^*(H^1_0)$.

 Note that any segment connecting $[D_1]$ to a rational point of the quadratic cone is contained in $\EffC(X)$. As the example is symmetric with respect to changing the first and the second coordinate, we obtain an effective integral divisor $D_2$ whose class is  
 $2H^0_1 - 2H^0_2 + 6H^0_3.$

 We expect that the pseudoeffective cone consists of the quadratic cone $\mathcal{C}_Q$ together with the $\Bir(X)$-orbit of the cones with vertices $\left[ D_1 \right]$ and $\left[ D_2 \right ]$, respectively glued tangentially to the cone as displayed in Figure \ref{figurePsef}.

\begin{figure}[h]
	\begin{overpic}[scale=.5]{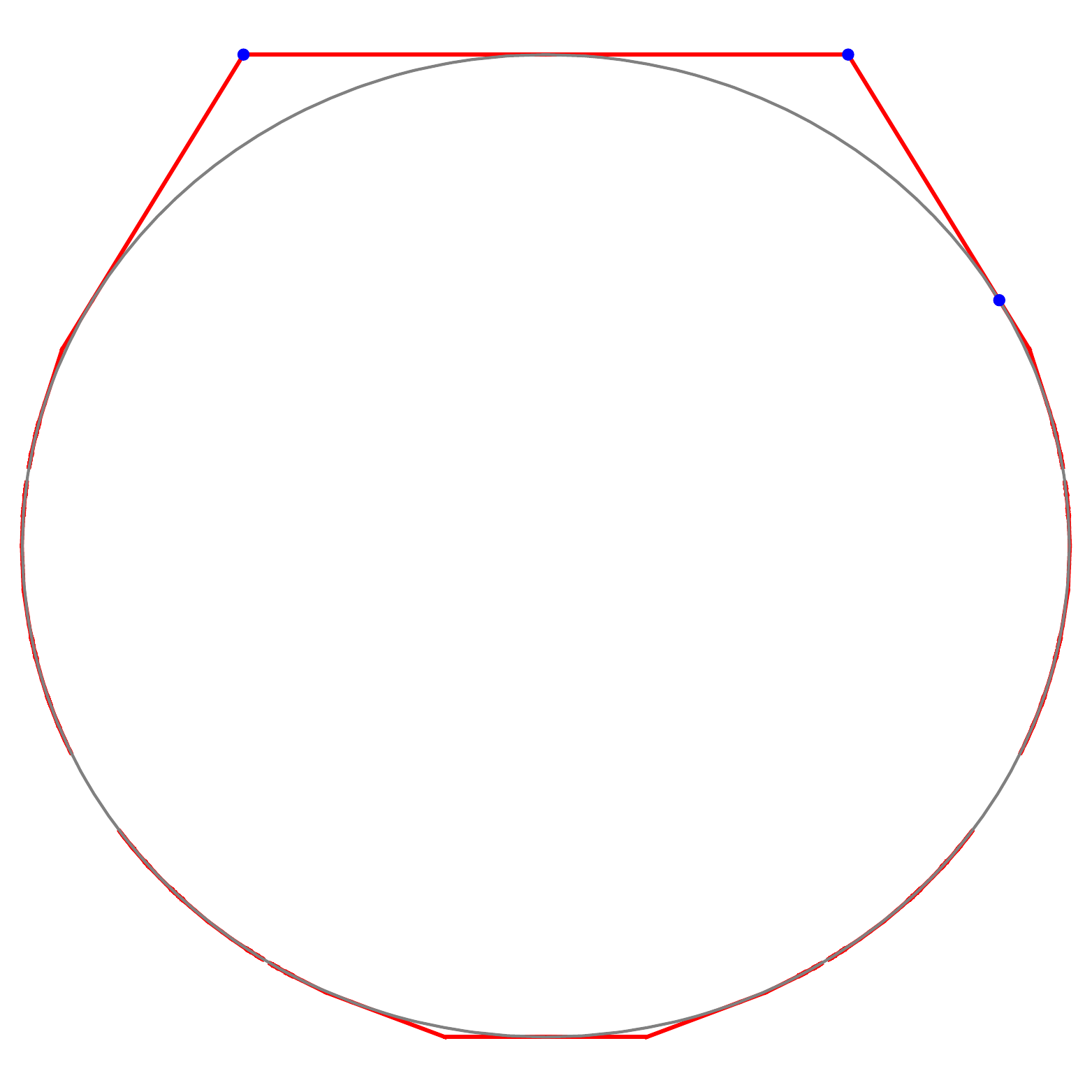}
 \put (80,95) {$[D_1]$}
 \put (9,95) {$[D_2]$}
  \put (93,73.5) {$\varphi_{0,1}^*H_0^1$}
  
	\end{overpic}
	\caption{The expected pseudoeffective cone of $X$: the quadratic cone $\mathcal{C}_Q$ together with the $\Bir(X)$-orbit (drawn in red) of the cone with vertex $\left[ D_1 \right]$ glued tangentially to the quadratic cone $\mathcal{C}_Q$.
	}
	\label{figurePsef}
\end{figure}
We close this example with a discussion on the numerical dimension of the pseudoeffective divisor $D_1$ (resp. $D_2$). 
\begin{definition}{(\cite[\S 5]{Nak04})}\label{defNumericalDimension}
	The numerical dimension $\kappa_\sigma(X,D)$ is the largest integer $k$ such that for some ample divisor $A$,  one has 
	$$
	\limsup_{m\to \infty} \frac{h^0(X,\lfloor mD\rfloor + A)}{m^k} > 0.
	$$
	If no such $k$ exists, we take $\kappa_\sigma(X,D) = -\infty$.
\end{definition}
Taking the supremum over all real numbers $k$, Lesieutre introduced the notion
$\kappa_\sigma^\RR$ and gave an example of a divisor $D$ on the boundary of the pseudoeffective cone of a Calabi-Yau threefold $X$ with $\kappa_\sigma(X,D) \neq \kappa_\sigma^\RR(X,D)$ (see \cite{L21}). 

We show that $D_1$ and $D_2$, along with the tangents to $C_Q$, are on the boundary of the pseudoeffective cone. In particular, this implies that the $\Bir(X)$-orbit of these segments are also on the boundary of $\EffC(X)$. 
\begin{proposition}
 
The boundary of the pseudoeffective cone $\EffC(X)$ contains the $\Bir(X)$-orbit of the line segment spanned by the divisor classes $[D_1]$ and $[D_2]$, and the segment spanned by the divisor classes $[D_1]$ and $\varphi_{0,1}^*H_0^1$.
\label{prop_numDim}
\end{proposition}

\begin{proof}
Note that any integral boundary divisor of $\MovC(X)\cap \partial \mathcal{C}_Q$ is not big. Indeed its numerical dimension $\kappa_\sigma$ is $2$ since its linear system corresponds to a projection of $X$ to $\PP^2$.  In particular,  the divisor $H^0_3$ is not big. 
By the properties of $\kappa_\sigma$ we know that 
\[
\kappa_\sigma(X,\alpha_1D_1 +\alpha_2D_2 )
\]
is independent of the positive numbers $\alpha_1,\alpha_2$. Hence, the numerical dimension is constant to $2= \kappa_\sigma(X,H^0_3)$ in this line segment which shows the claim.

Now consider a divisor class $[D]$ on the segment spanned by $[D_1]$ and $\varphi_{0,1}^*H_0^1$, and assume it is not on the boundary of $\EffC(X)$. Therefore, for an ample divisor $A$ and $0 <\varepsilon \ll 1$ we have that $D - \varepsilon A$ is big. Consider the line $L$ connecting $D - \epsilon A$ and $\varphi_{0,1}^*H_0^1$. Because the segment spanned by $[D_1]$ and $\varphi_{0,1}^*H_0^1$ is tangent to $C_Q$, there exists a divisor class $E$ on the line $L$ in the interior of $C_Q$ such that $E$ is big. This implies that the divisor class $\varphi_{0,1}^*H_0^1$ is also big, which is a contradiction. Therefore, the segment spanned by $[D_1]$ and $\varphi_{0,1}^*H_0^1$ is on the boundary of $\EffC(X)$.
\end{proof}

\begin{remark}
By \cite[\S 5, Proposition 2.7]{Nak04} we have 
\[0  \leq \kappa_\sigma(X,D_1)= \kappa_\sigma(X,D_2) \leq \kappa_\sigma(D_1+D_2) = 2.
\]  
\end{remark}

On Example \ref{examp_3foldn2m3} there are three types of divisors on the boundary of $\MovC(X)$: rational points, eigenvectors of elements of $\Bir(X)$, and accumulation points of eigenvectors. The numerical dimension of the first two is well understood. If the divisor class corresponds to a rational point on the boundary, then it can be pulled back birationally to a rational nef divisor on some minimal model $X_i$, and hence it has an integral numerical dimension $\kappa_\sigma^\RR = \kappa_\sigma$.

In the case of an eigenvector of an element $f \in \Bir(X)$, one can perform a similar computation as done in \cite{L21} and \cite{HS22} by restricting to the two-dimensional cone generated by the eigenvectors associated to the unique eigenvalue $\lambda > 1$ of $f$ and $f^{-1}$. In this case we also obtain that $\kappa_\sigma^\RR = 3/2$.

The third case, corresponding to accumulation points, is not well understood and could be a source of interesting numerical dimensions.

Example \ref{example_nontrivialaut} provides a new interesting divisor class, specifically $D_1$ and its $\Bir(X)$-orbit. Computing its numerical dimension requires a different method than the ones described above, and given that $D_1$ is an integral divisor class, it could be another potential source for an interesting behavior of the numerical dimension $\kappa_\sigma^\RR$.

\begin{question} What is $\kappa_\sigma ^\RR(X,D_1)$ and does $\kappa_\sigma ^\RR(X,D_1) = \kappa_\sigma (X,D_1)$ hold?
\label{question_numDim}
\end{question}	

\begin{example}[$m=2$, $n=4$, $\Aut(X)$ non-trivial]\label{examp_trivialActing}
As the last example we present a Calabi-Yau threefold $X_0 \subset \PP^4_1 \times \PP^4_2$ with $\rho(X_0)=2$ which has a non-trivial automorphism group $\Aut(X_0)$ acting trivially on $N^1(X_0)_\RR$. For the sake of simplicity, we denote by $x_i$ the variables of $\PP_0^4$, by $y_j$ the ones of $\PP_1^4$ and by 
$z_k$ the ones of $\PP_2^4$. Let us consider the element $  g \in \Aut(\PP^4_0 \times \PP^4_1 \times \PP^4_2)$ given by
\[
g  =\big( \begin{pmatrix}
-1 & & & &  \\  & -1& & &  \\ & & 1& & \\ & & & 1 &  \\ & & &  & 1
\end{pmatrix} , 
\begin{pmatrix}
-1 & & & &  \\  & -1& & &  \\ & & 1& & \\ & & & 1 &  \\ & & &  & 1
\end{pmatrix},
\begin{pmatrix}
1 & & & &  \\  & 1& & &  \\ & & 1& & \\ & & & -1 &  \\ & & &  & -1
\end{pmatrix}
\big). 
\] 
A basis of all $(1,1,1)$-forms invariant under $g$ is given by 63 monomials $x_iy_jz_k$ with 
\begin{center}
\begin{tabular}{c|c|c| c}
	 $i$  & $j$ & $k$ & number\\ \hline \hline
	$0,1$ &  $0,1$ & $0,1,2$ & $2\cdot2\cdot 3 = 12 $ \\ \hline
		$0,1$ &  $2,3,4$ & $3,4$ & $12$ \\ \hline
		$2,3,4$ &  $0,1$ & $3,4$ & $12 $\\ \hline
		$2,3,4$ &  $2,3,4$ & $0,1,2$ & $27$ \\ \hline \hline
		& & & $63$
		\end{tabular}.
\end{center}
Taking any combination of these monomials, we obtain an invariant $(1,1,1)$-form $F$ and by differentiating with respect to the $x_i$ we obtain a variety $X_0 \subset \PP^4_1 \times \PP^4_2$ which is invariant under the induced action on $\PP^4_1 \times \PP^4_2$.  We verified computationally with $\textit{Macaulay2}$ (see \cite{CYCIM2}) that a random combination of the basis elements yields a non-singular complete intersection Calabi-Yau threefold $X_0$ with minimal models $X_1,X_2$ and $X_3$ which are also non-singular. 

The space of all possible invariant forms $F$ is of dimension at most
\[ 63 + 3\cdot \dim(\PGL_5(\CC)) - 1 = 134 \] so we do not get a useful dimension bound as in Proposition \ref{prop_AutTrivial} because the 
(projective) space of all possible $(1,1,1)$-forms $F$ is of dimension $124$. 
\end{example}
\bibliographystyle{amsalpha}

\end{document}